\documentclass[12pt]{amsart}

\usepackage{amsfonts,amsmath,amsthm,amssymb,latexsym,amsthm,newlfont,enumerate,color,tikz-cd}

\newif\ifslide

\theoremstyle{plain}
\newtheorem{theorem}{Theorem}[section]

\newtheorem{corollary}[theorem]{Corollary}
\newtheorem{lemma}[theorem]{Lemma}

\newtheorem{proposition}[theorem]{Proposition}
\newtheorem{definition-lemma}[theorem]{Definition-Lemma}

\newtheorem{red-question}[theorem]{\textcolor{red}{Question}}

\newtheorem{conjecture}[theorem]{Conjecture}

\theoremstyle{definition}
\newtheorem{definition}[theorem]{Definition}
\newtheorem{remark}[theorem]{Remark}

\def\ideal#1.{I_{#1}}
\def\ring#1.{\mathcal {O}_{#1}}

\def\RR{\mathbb R}

\def\QQ{\mathbb Q}

\def\fring#1.{\hat{\mathcal {O}}_{#1}}
\def\proj#1.{\mathbb {P}(#1)}
\def\pr #1.{\mathbb {P}^{#1}}
\def\dpr #1.{\hat{\mathbb {P}}^{#1}}
\def\af #1.{\mathbb A^{#1}}
\def\Hz #1.{\mathbb F_{#1}}
\def\Hbz #1.{\overline{\mathbb F}_{#1}}
\def\fb#1.{\underset #1 {\times}}
\def\rest#1.{\underset {\ \ring #1.} \to \otimes}
\def\au#1.{\operatorname {Aut}\,(#1)}
\def\deg#1.{\operatorname {deg } (#1)}

\def\pic#1.{\operatorname {Pic}\,(#1)}
\def\pico#1.{\operatorname{Pic}^0(#1)}
\def\picg#1.{\operatorname {Pic}^G(#1)}
\def\ner#1.{NS (#1)}
\def\rdown#1.{\llcorner#1\lrcorner}
\def\rfdown#1.{\lfloor{#1}\rfloor}
\def\rup#1.{\ulcorner{#1}\urcorner}
\def\rcup#1.{\lceil{#1}\rceil}

\def\n1#1.{\operatorname {N_1}(#1)}  
\def\cn1#1.{\overline{\operatorname {N^1}(#1)}} 
\def\cone#1.{\operatorname {NE}(#1)}     
\def\ccone#1.{\overline{\operatorname {NE}}(#1)}
\def\none#1.{\operatorname {NF}(#1)}
\def\cnone#1.{\overline{\operatorname {NF}}(#1)}
\def\mone#1.{\operatorname {NM}(#1)} 
\def\cmone#1.{\overline{\operatorname {NM}}(#1)}

\def\coef#1.{\frac{(#1-1)}{#1}}
\def\vit#1.{D_{\langle #1 \rangle}}
\def\mm#1.{\overline {M}_{0,#1}}
\def\H1#1.{H^1(#1,{\ring #1.})}
\def\ac#1.{\overline {\mathbb F}_{#1}}

\def\adj#1.{\frac {#1-1}{#1}}
\def\spn#1.{\overline{#1}}
\def\pek#1.#2.{\Cal P^{#1}(#2)}
\def\plk#1.#2.{\Cal P^{\leq #1}(#2)}
\def\ev#1.{\operatorname{ev_{#1}}}
\def\ilist#1.{{#1}_1,{#1}_2,\dots}
\def\bminv#1.{(\nu_1,s_1;\nu_2,s_2;\dots ;\nu_{#1},s_{#1};\nu_{r+1})}
\def\zinv#1.{(\nu_1,s_1;\nu_2,s_2;\dots ;\nu_{#1},s_{#1};0)}
\def\iinv#1.{(\nu_1,s_1;\nu_2,s_2;\dots ;\nu_{#1},s_{#1};\infty)}

\def\scr #1.{\mathcal #1}

\def\llist#1.#2.{{#1}_1,{#1}_2,\dots,{#1}_{#2}}
\def\ulist#1.#2.{{#1}^1,{#1}^2,\dots,{#1}^{#2}}
\def\lomitlist#1.#2.{{#1}_1,{#1}_2,\dots,\hat {{#1}_i}, \dots, {#1}_{#2}}
\def\lomitlistz#1.#2.{{#1}_0,{#1}_1,\dots,\hat {{#1}_i}, \dots, {#1}_{#2}}
\def\loc#1.#2.{\Cal O_{#1,#2}}
\def\fderiv#1.#2.{\frac {\partial #1}{\partial #2}}
\def\deriv#1.#2.{\frac {d #1}{d #2}}

\def\map#1.#2.{#1 \longrightarrow #2}
\def\rmap#1.#2.{#1 \dasharrow #2}
\def\emb#1.#2.{#1 \hookrightarrow #2}
\def\non#1.#2.{\text {Spec }#1[\epsilon]/(\epsilon)^{#2}}
\def\Hi#1.#2.{\text {Hilb}^{#1}(#2)}
\def\sym#1.#2.{\operatorname {Sym}^{#1}(#2)}
\def\Hb#1.#2.{\text {Hilb}_{#1}(#2)}
\def\Hm#1.#2.{\Hom_{#1}(#2)}
\def\prd#1.#2.{{#1}_1\cdot {#1}_2\cdots {#1}_{#2}}
\def\Bl #1.#2.{\operatorname {Bl}_{#1}#2}
\def\pl #1.#2.{#1^{\otimes #2}}
\def\mgn#1.#2.{\overline {M}_{#1,#2}}
\def\ialist#1.#2.{{#1}_1 #2 {#1}_2, #2\dots}
\def\pair#1.#2.{\langle #1, #2\rangle}
\def\vandermonde#1.#2.{\left|
\begin{matrix}
1 & 1 & 1 & \dots & 1\\
{#1}_1 & {#1}_2 & {#1}_3 & \dots & {#1}_{#2}\\
{#1}_1^2 & {#1}_2^2 & {#1}_3^2 & \dots & {#1}_{#2}^2\\
\vdots & \vdots & \vdots & \ddots & \vdots\\
{#1}_1^{#2-1} & {#1}_2^{#2-1} & {#1}_2^{#2-1} & \dots & {#1}_{#2}^{#2-1}\\
\end{matrix}
\right|
}
\def\vandermondet#1.#2.{\left|
\begin{matrix}
1 & {#1}_1   & {#1}_1^2 & \dots & {#1}_1^{#2-1}\\
1 & {#1}_2   & {#1}_2^2 & \dots & {#1}_2^{#2-1}\\
1 & {#1}_3   & {#1}_3^2 & \dots & {#1}_3^{#2-1}\\
\vdots & \vdots & \vdots & \ddots & \vdots\\
1 & {#1}_{#2}& {#1}_{#2}^2 & \dots & {#1}_{#2}^{#2-1}\\
\end{matrix}
\right|
}
\def\gr#1.#2.{\mathbb{G}(#1,#2)}

\def\alist#1.#2.#3.{{#1}_1 #2 {#1}_2 #2\dots #2 {#1}_{#3}}
\def\zlist#1.#2.#3.{#1_0 #2 #1_1 #2\dots #2 #1_{#3}}
\def\lomitlist30#1.#2.#3.{{#1}_0,{#1}_1 #2 \dots #2\hat {{#1}_i} #2\dots #2 {#1}_{#3}}
\def\lmap#1.#2.#3.{#1 \overset{#2}{\longrightarrow} #3}
\def\mes#1.#2.#3.{#1 \longrightarrow #2 \longrightarrow #3}
\def\ses#1.#2.#3.{0\longrightarrow #1 \longrightarrow #2 \longrightarrow #3 \longrightarrow 0}
\def\les#1.#2.#3.{0\longrightarrow #1 \longrightarrow #2 \longrightarrow #3}
\def\res#1.#2.#3.{#1 \longrightarrow #2 \longrightarrow #3\longrightarrow 0}
\def\Hi#1.#2.#3.{\text {Hilb}^{#1}_{#2}(#3)}
\def\ten#1.#2.#3.{#1\underset {#2}{\otimes} #3}
\def\lomitlist30#1.#2.#3.{{#1}_0 #2 {#1}_1 #2 \dots #2 \hat {{#1}_i} #2 \dots #2 {#1}_{#3}}
\def\mderiv#1.#2.#3.{\frac {d^{#3} #1}{d #2^{#3}}}

\def\Hom{\operatorname{Hom}}

\def\Proj{\operatorname{Proj}}

\def\dim{\operatorname{dim}}

\def\deg{\operatorname{deg}}

\def\det{\operatorname{det}}

\def\ker{\operatorname{Ker}}

\def\rk{\operatorname{rk}}
\def\im{\operatorname{Im}}

\def\rest{\operatorname{res}}
\def\Sym{\operatorname{Sym}}

\def\rar{\rightarrow}
\def\drar{\dashrightarrow}

\def\e{\Cal E}

\def\e1{E_1}
\def\e2{E_2}

\def\bM#1.{\mathbf{M}_{#1}}

\usepackage{hyperref,amscd,mathtools}

\DeclareOldFontCommand{\rm}{\normalfont\rmfamily}{\mathrm}

\newcommand{\marg}[1]{\normalsize{{\color{red}\footnote{{\color{blue}#1}}}{\marginpar[{\color{red}\hfill\tiny\thefootnote$\rightarrow$}]{{\color{red}$\leftarrow$\tiny\thefootnote}}}}}

\newcommand{\Cal}[1]{\marg{(Calum) #1}}

\def\K#1.{K_{#1}}
\def\subs#1.{_{#1}}					
\def\sups#1.{^{#1}}	

\def \QP {\mathbb Q_{>0}}

\title{On semi-ampleness of the moduli part}

\thanks{SF and CS were partially supported by ERC starting grant \#804334.  CS was partially supported by EPSRC}

\author[S.~Filipazzi]{Stefano Filipazzi}
\email{stefano.filipazzi@duke.edu}
\address{
Department of Mathematics, Duke University, 120 Science Drive, 117 Physics Building, Campus Box 90320, Durham, NC 27708-0320, USA
}

 \author[C. ~Spicer]{Calum Spicer}
\email{calum.spicer@kcl.ac.uk}
\address{Department of Mathematics, King's College London, Strand,
London WC2R 2LS, UK}

\begin{document}
\begin{abstract}We discuss a conjecture of Shokurov on the semi-ampleness of 
the moduli part of a general fibration.
\end{abstract}

\maketitle
\tableofcontents


\section{Introduction}

In \cite{shokurov2021log}, 
the notion of the moduli and discriminant parts
of a generically log canonical (GLC) fibration $(X/Z, B)$ is defined.
By a GLC fibration we mean the data of a log pair $(X, B)$
and a contraction $f\colon X \rightarrow Z$, i.e., a projective
 surjective morphism with connected fibres,
between normal quasi-projective varieties, such that $(X, B)$ is log canonical above the generic point of $Z$.
We may define a discriminant divisor $B_Z$ on $Z$ which (roughly) measures the singularities of the fibres of $(X, B)$
over $Z$, and the moduli part is then defined as $M_X \coloneqq (K_X+B) - f^*(K_Z+B_Z)$, we refer to \S~\ref{s_prop_*} for precise definitions.
These definitions are straightforward extensions
of the corresponding notions for LC trivial fibrations. 

For GLC fibrations, it is known that, after passing to a sufficiently high model, the moduli part
of the fibration is compatible with pull-backs.  Moreover, in the case of an LC trivial fibration,  
it is expected that the moduli part becomes semi-ample on such a model.
This conjecture is known as the b-semi-ampleness conjecture, see \cite[Conjecture 7.13.1]{ps09}.
Analogously, one might hope that a similar statement holds for a general GLC fibration, namely, 
after replacing $(X/Z, B)$ by a birationally equivalent pair $(X'/Z', B')$,
the moduli part $M_{X'}$ of $(X'/Z', B')$ becomes semi-ample.

However, as observed by Keel, see \cite[\S~5.4]{ACSS}, the example of $f\colon (C \times C, \Delta) \rightarrow C$
where $C$ is a curve of genus $\geq 2$ and $\Delta$ is the diagonal shows that such a statement does not hold.
In light of this and other examples, Shokurov conjectured (\cite[Conjecture 1]{shokurov2021log})
that up to replacing $(X/Z, B)$ by an appropriate model,
$M_X$ becomes semi-ample after any small perturbation 
by an ample divisor coming from the moduli space of fibres.
We consider here a variant on Shokurov's conjecture, where, by analogy with the moduli space of stable pairs, we encode the notion of polarisation on a suitable moduli space of the fibres into the determinant of a twisted version of the Hodge bundle.

\begin{conjecture}
\label{conj_shokurov}
Suppose that $(X, B)$ is log canonical, $B \geq 0$, $(X/Z, B)$ has maximal moduli
(see \cite[Definition 2.20]{ACSS}),
	$M_X$ is $f$-nef and BP stable (see \cite[Definition 2.5]{ACSS}).
Then the following hold

	\begin{enumerate}
			\setcounter{enumi}{0}
	\item for all $m \gg 0$, $\det f_*\mathcal O_X(mM_X)$ is semi-ample;
	
	\item for all $m \gg 0$ and any  $\epsilon \in \QP$, 
$M_X+\epsilon f^*c_1( \det f_*\mathcal O_X(mM_X))$ is semi-ample;

\item $\kappa(M_X) \geq \kappa(X/Z, B)+{\rm var}(X/Z, B)$.
\end{enumerate}
\end{conjecture}

We note that if $B = 0$, $X$ has canonical singularities and $f\colon X \rightarrow B$ is a family
of good minimal models, then item (2) was proven in \cite{Kawamata85a}.
In \cite{KP17}, Kov\'{a}cs and Patakfalvi settled item (2) 
when $K_X+B$ is $f$-big and the generic fibre has klt singularities.
In the same work, en route to proving the projectivity of the moduli 
space of stable pairs, 
they settled item (1) when $f$ is a family of stable varieties.

We observe that the hypotheses of Conjecture~\ref{conj_shokurov} are natural.
Indeed, under the assumption of the standard MMP conjectures (which are settled in relative dimension at most 3), given any fibration $f \colon (X,B) \rightarrow Z$ whose generic fibre is log canonical and has non-negative Kodaira dimension, $f$ can then be birationally modified to satisfy the hypotheses of Conjecture~\ref{conj_shokurov}.
Therefore, a fibration as in Conjecture~\ref{conj_shokurov} should be thought as a distinguished preferred birational model of an arbitrary fibration.

Our first main theorem partially settles Conjecture~\ref{conj_shokurov} over 1-dimesional bases 
by using ideas from foliation theory, from \cite{MR3959071}
and the Simpson correspondence.
The following is proven in Theorem \ref{thm_curve_base}.

\begin{theorem}
Let $f\colon (X, B) \rightarrow Z$ be a GLC fibration between projective varieties
with $B \geq 0$, $\dim X= n$ and $\dim Z = 1$.
Suppose that
$(X, B)$ is klt over the generic point of $Z$ and the generic fibre of $f$ admits a good minimal model.

	Then we may find a birationally equivalent GLC fibration $f'\colon (X', B') \rightarrow Z$ with moduli part $M'$ such that $K_{X'}+B'$ is nef over $Z$ and 
	\begin{enumerate}
\item $\kappa(M') \geq \kappa(X'/Z, B')+{\rm var}(X'/Z, B')$; 

	\item if, in addition, $f\colon (X, B^h) \to Z$ is locally stable, then for all $m \gg 0$, $\lambda'_m \coloneqq \det f'_*\mathcal O_{X'}(mM')$ is semi-ample; and
	
	\item if, in addition, $f\colon (X, B^h) \to Z$ is locally stable, $M'+\epsilon f'^*c_1(\lambda_m')$ is semi-ample for $m\gg 0$ and any $\epsilon \in \QP$.


\end{enumerate}
\end{theorem}

Our next theorem is that items (1) and (2) of Conjecture \ref{conj_shokurov} 
holds for locally stable morphisms whenever $\dim X -\dim Z \leq 2$.

\begin{theorem}\label{intro_thm_surfaces}
Let $f\colon (X, B) \rightarrow Z$ be a GLC fibration between projective varieties which is also a locally stable family
with $B \geq 0$, $\dim X= n$ and $\dim Z \geq n-2$.
Suppose that a general fibre of $f$ has non-negative Kodaira dimension.

        Then we may find a birationally equivalent GLC fibration $f'\colon (X', B') \rightarrow Z'$
        such that $K_{X'}+B'$ is nef over $Z'$ and
	\begin{enumerate}
		\item $\lambda'_m$ is semi-ample for $m \gg 0$, where $M'$ is the moduli part of $(X', B')\rightarrow Z'$ and  $\lambda_m' \coloneqq \det f'_*\mathcal O_{X'}(mM')$; and 
		\item $M'+\epsilon f'^*c_1(\lambda_m')$ is semi-ample for $m\gg 0$ and any $\epsilon \in \QP$.
	\end{enumerate}
	\end{theorem}

\begin{remark}\label{rmk_intro}
More generally, Theorem~\ref{intro_thm_surfaces} holds true for fibrations whose general fibre has Kodaira dimension at least $\dim X - \dim Z - 2$.
\end{remark}

In fact, Theorem~\ref{intro_thm_surfaces} and Remark~\ref{rmk_intro} follow from a more general statement, together with the proof of the b-semi-ampleness
conjecture in relative dimension at most $2$, see \cite{ps09,Fujino03,Fil20b,Gio22}.
In the case of an LC trivial fibration items (1) and (2) of Conjecture~\ref{conj_shokurov}
reduce to the b-semi-ampleness conjecture. 
Our next theorem, which then specialises to Theorem~\ref{intro_thm_surfaces}, shows that for locally stable families the general
case is implied by the b-semi-ampleness conjecture.  We refer to  
\cite[Conjecture 7.13.1]{ps09} for a precise statement of the b-semi-ampleness conjecture.

\begin{theorem}[ = Theorem \ref{thm_*}]\label{main_thm_intro}

Assume the b-semi-ampleness conjecture holds for LC trivial fibrations of relative dimension at most $n-1$.
Let $f\colon (X, B) \rightarrow Z$ be a GLC fibration between projective varieties which is also a locally stable family with $B \geq 0$ and $\dim X= n$.
Suppose that a general fibre of $f$ admits a good minimal model.

Then we may find a birationally equivalent GLC fibration 
$f'\colon (X', B') \rightarrow Z'$
fitting into the following commutative diagram
\[
\begin{tikzcd}
  (X', B') \arrow[r, "\alpha", dashrightarrow] \arrow[d, "f'"] & (X, B) \arrow[d, "f"] \\
  Z' \arrow[r, "\beta"] &  Z
 \end{tikzcd}
\]
such that the following hold, where $M'$ is the moduli part of $(X'/Z', B')$

\begin{enumerate}
\item $(X'/Z', B')$ is BP stable and has maximal moduli;

\item $\lambda_m'$ is semi-ample for $m\gg 0$,
where $\lambda_m' = \det f'_*\mathcal O_{X'}(mM')$;

\item $M'+\epsilon f'^*c_1(\lambda_m')$ is semi-ample for $m\gg 0$ and any $\epsilon \in \QP$; and

\item $M'$ and $\lambda_m'$ are compatible with base change in an appropriate sense.
\end{enumerate}

\end{theorem}

We refer to Theorem \ref{thm_*} for a precise (and slightly more general) statement.
We remark that Shokurov has shown a related statement, cf. \cite[Corollary 2]{shokurov2021log}.

\subsection*{Acknowledgements}
We thank Zsolt Patakfalvi for many insightful discussions. 
We also thank Paolo Cascini, Giulio Codogni, 
Christopher D. Hacon, Giovanni Inchiostro, Jihao Liu, 
Roberto Svaldi and
Luca Tasin
for many useful discussions. We thank Jihao Liu for feedback on a draft version of this work.
Lastly, we thank the anonymous referee for helpful comments that improved the clarity of this work.


\section{Preliminaries}

We work over $\mathbb C$.
We refer to \cite{KM98} for the standard terminology in birational geometry.
For the language of generalised pairs and b-divisors, we refer to \cite{FS20}.

Our varieties are connected and quasi-projective
unless otherwise stated.
Given a normal variety $X$ and an open subset $U \subset X$, we say that $U$ is big if $\mathrm{codim}_X(X \setminus U) \geq 2$.

All divisors have coefficients in $\QQ$.
Unless otherwise stated, a divisor means a Weil $\QQ$-divisor.
Given a dominant morphism $f \colon X \rar Y$ and a divisor $D=\sum a_i P_i$, we define its horizontal part 
$D^h$ and its vertical part $D^v$ as
\[
D^h \coloneqq \sum_{i|P_i\;\mathrm{dominates}\;Y}a_i P_i, \quad D^v \coloneqq D-D^h.
\]

A log pair $(X, B)$ consists in a normal variety $X$ and a divisor $B$ such that $K_X+B$ is $\mathbb Q$-Cartier.

Given a coherent sheaf $\mathcal F$ of rank $r$ on a normal variety we define 
$\det \mathcal F \coloneqq (\bigwedge^r \mathcal F)^{**}$.

\subsection{(Locally) stable families}

We recall the definition of locally stable family of pairs over a reduced base as in 
\cite[Definition-Theorem 4.7]{Kollar21}.

\begin{definition} \label{def_stable}
Let $S$ be a reduced scheme, $f\colon X \rightarrow S$ a flat morphism of finite type and 
$f\colon (X, \Delta) \rightarrow S$ a well defined family of pairs (see \cite[Definition-Theorem 4.7]{Kollar21}).
Assume that $(X_s, \Delta_s)$ is slc for every $s \in S$.  Then $f\colon (X, \Delta) \rightarrow S$
is locally stable if the following equivalent conditions holds.

\begin{itemize}
\item[(1)] $\K X/S. + \Delta$ is $\RR$-Cartier.
\item[(2)] $f_T \colon (X_T,\Delta_T) \rar T$ is locally stable whenever $T$ is the spectrum of a DVR and 
$q \colon T \rar S$ is a morphism (see \cite[Definition-Theorem 2.3]{Kollar21} for the notion of a locally stable
family over a DVR).
\item[(3)] There is a closed subset $Z \subset X$ such that ${\rm codim} (Z \cap X_s, X_s) \geq 3$ for 
all $s \in S$ and $f_{X \setminus Z}\colon (X\setminus Z) \rightarrow S$ satisfies the above (1-2). 
\end{itemize}
Such a family is called stable if, in addition, $\K X/S. + \Delta$ is $f$-ample.
\end{definition}

By \cite[Theorem 4.8]{Kollar21} 
if $f \colon (X,\Delta) \rar S$ is a (locally) stable family over a reduced base $S$ 
and $\phi \colon W \rar S$ is a morphism with $W$ reduced, then the family over $W$ 
obtained by fibre product is again a (locally) stable family.

We recall the following criterion for local stability.
\begin{proposition}
\label{prop_kollar_criterion}
Let $S$ be a smooth variety and $p\colon (X, \Delta) \rightarrow S$ a projective morphism with $\Delta \geq 0$.
Then $p\colon (X, \Delta) \rightarrow S$ is locally stable if and only if the pair $(X, \Delta+p^*D)$
is semi-log canonical for every reduced simple normal crossing divisor $D \subset S$.
\end{proposition}
\begin{proof}
This follows directly from
\cite[Corollary 4.55]{Kollar21}.
\end{proof}

\subsection{Property $(*)$, BP stable and all that}
\label{s_prop_*}
We recall some definitions from \cite{ACSS}.

Let $f\colon X \rightarrow Z$ be a projective morphism between normal varieties
and let $(X, B)$ be a log pair.
We say that $f\colon (X, B) \rightarrow Z$ (or $(X/Z, B)$ when the morphism $f$ is understood)
is {\bf generically log canonical} or GLC, if $Z$ is irreducible and $(X, B)$ is log canonical over the generic point 
of $Z$.
In this case, we say that $(X/Z,B)$ is a GLC pair induced by the morphism $f$.
When $f\colon X\rightarrow Z$ is a contraction, i.e., $f_*\cal O_X= \cal O_Z$, we will refer to $f$ 
as a GLC contraction.
When $f$ is a GLC contraction and $\dim X > \dim Z$, we will refer to $f$ as a GLC fibration.

Let $f\colon (X, B) \rightarrow Z$ be GLC. For any prime divisor $P\subset Y$, we define
\[
\gamma_P \coloneqq \sup\{t\in \mathbb R\mid (X,B+tf^*P) \text{ is log canonical over the generic point of }P\}.
\]
The {\bf discriminant } of $(X/Z,B)$ is the $\mathbb R$-divisor
\[
B_Z \coloneqq \sum_P (1-\gamma_P)P.
\]

We now recall the definition of the moduli part of a GLC fibration  $f\colon (X, B) \rightarrow Z$.
We first assume  that $f$ is equidimensional and
$K_Z+B_Z$ is $\mathbb R$-Cartier.
In this case, we define the moduli part $M_X$ to be
\[
M_X\coloneqq K_X+B-f^*(K_Z+B_Z).
\]

Consider now any GLC fibration $f\colon (X, B) \to Z$. By \cite[Theorem 2.1]{AK00}, 
there exists an equidimensional contraction $f'\colon X'\to Z'$ which is birationally equivalent to $f$, 
such that $Z'$ is smooth and the induced maps $\alpha\colon Z'\to Z$ and $\beta\colon X'\to X$ 
are birational morphisms.
Let $M_{X'}$ be the moduli part of the induced morphism $(X',B') \rightarrow Z'$. 
We then define the {\bf moduli part} of $f\colon (X, B) \rightarrow Z$  as
\[M_X\coloneqq\alpha_*M_{X'}.
\]
It is easy to check that $M_X$ does not depend on the choice of $f'$.

We refer to \cite[Definition 2.5]{ACSS} for the notion of BP stability and to \cite[Definition 2.20]{ACSS}
for the definition of maximal moduli.

\begin{definition}\label{d_property*}
Let $(X/Z,B)$ be a GLC pair induced by a  morphism $f\colon X\to Z$.  
We say that $(X/Z,B)$ satisfies {\bf Property $(*)$} if $f$ is a projective  contraction and the following hold:
\begin{enumerate}
\item  there exists a reduced divisor $\Sigma_Z$ on $Z$  such that $(Z,\Sigma_Z)$ is log smooth and 
the vertical part of $B$ is a reduced divisor that coincides with 
$f^{-1}(\Sigma_Z)$; and 
\item for any closed point $z\in Z$ and for any  $\Sigma\ge \Sigma_Z$ reduced 
divisor on $Z$ such that $(Z,\Sigma)$ is log smooth around $z$,  
we have that $(X,B+f^*(\Sigma-\Sigma_Z))$ is log canonical around $f^{-1}(z)$.
\end{enumerate}
\end{definition}

Given a generalised pair $(X, B,\bM.)$ and a contraction $f\colon X \rightarrow Z$ such that $(X, B,\bM.)$ is generalised log canonical
above the generic point of $Z$ we say that $(X/Z, B,\bM.)$ is a {\bf generalised GLC pair}.
If a generalised GLC pair satisfies the conditions of Definition \ref{d_property*} 
for a generalised pair, mutatis mutandis, we will
say that $(X/Z, B,\bM.)$ satisfies {\bf generalised Property $(*)$}.

\begin{lemma}
\label{lem_img_lc_centre_is_lc}
Let $f\colon (X, B, \bM.)\rightarrow Z$ be a generalised GLC contraction which satisfies generalised Property $(*)$
and let $(Z, \Sigma_Z)$ be the associated log smooth pair.
Let $W$ a generalised log canonical centre of $(X, B, \bM.)$.  Then $f(W)$ is a log canonical 
centre of $(Z, \Sigma_Z)$.  
\end{lemma}
\begin{proof}
Follows directly from the definition of generalised Property $(*)$.
\end{proof}
We recall that given a generalised pair $(X, B, \bM.)$ that $X$ itself is a log canonical centre.

\begin{lemma} \label{good_mm_fibre_to_space}
Let $f \colon (X,B) \rar Z$ be a GLC contraction where $B \geq 0$.
Assume that $(X,B)$ is log canonical and that every log canonical centre of $(X,B)$ dominates $Z$.
Then, if a general fibre of $f$ has a good minimal model, then $(X,B)$ has a good minimal model over $Z$.
\end{lemma}

\begin{proof}
By \cite[Theorem 2.9]{LT19}, we may show the claim up to extracting some log canonical places of $(X,B)$.  Therefore, we are free to replace $(X, B)$ by a dlt modification, which is guaranteed to exist by \cite[Theorem 3.1]{KK09}, and so we may assume that $(X, B)$ is dlt.
Furthermore, since all log canonical places are horizontal over $Z$, by \cite[Theorem 1.1]{HX11}, it suffices to show the claim over a non-empty open subset of $Z$.
Therefore, throughout the proof we are free to shrink $Z$ as needed.

Let $\pi \colon X' \rar X$ be a log resolution of $(X,B)$ that only extracts valuations with positive log discrepancy, which exists by \cite[Theorem 2.44]{KM98}.
We denote the strict transform of $B$ by $B'$, and we let $E'$ be the reduced $\pi$-exceptional divisor.
Up to shrinking $Z$, we may assume that $(X',B'+E') \rar Z$ is log smooth.
By construction, a good minimal model for $(X,B)$ over $Z$ is also a good minimal model for $(X',B'+E')$ over $Z$, and vice versa.

By assumption, for $z \in Z$ general, $(X_z,B_z)$ has a good minimal model.
Since $\pi$ is a fibrewise log resolution that only extracts valuations of positive log discrepancy, this good minimal model is also a good minimal model for $(X'_z,B'_z+E'_z)$.
Then, by \cite[Theorem 1.2]{HMX14}, $(X',B'+E')$ has a good minimal over $Z$.
In turn, this is a good minimal model for $(X,B)$ over $Z$, and the claim follows.
\end{proof}

\begin{lemma}
\label{lem_gen_fibre_min_model}
Let $f\colon (X, B)\rightarrow Z$ be a GLC contraction where $B \geq 0$ and 
which satisfies Property $(*)$.
Suppose that a general fibre of $(X, B) \rightarrow Z$ admits a good minimal model.  
Then $(X, B)$ has a good minimal
model over $Z$.
\end{lemma}

\begin{proof}
By Lemma \ref{lem_img_lc_centre_is_lc}, we see that for $0<t\ll 1$ every log canonical centre of $(X, B-tf^*\Sigma)$ 
dominates $Z$ and $B-tf^*\Sigma \geq 0$.
Then, by Lemma~\ref{good_mm_fibre_to_space}, $(X, B-tf^*\Sigma)$ admits a good minimal model over $Z$.
Since $K_{X}+B \sim_{\mathbb Q} K_{X}+B-tf^*\Sigma$ we see that this is a 
good minimal model for $(X, B)$ over $Z$ as well.
\end{proof}

\begin{lemma}
\label{lem_prop*_horz_semi-stable}
Let $f\colon (X, B)\rightarrow Z$ be a GLC contraction which satisfies Property $(*)$ 
and such that $B=B^h \geq 0$.
Then $f\colon (X, B)\rightarrow Z$ is a locally stable family.
Conversely, any locally stable family over a smooth base satisfies Property $(*)$.
\end{lemma}
\begin{proof}
From the equality $B = B^h$ we see that if $\Sigma_Z$ is the reduced divisor associated
to $(X/Z, B)$ then $\Sigma_Z = 0$.  The definition of Property $(*)$ together 
with Proposition \ref{prop_kollar_criterion}
immediately implies our claim.
\end{proof}

\subsection{Preliminaries on foliations}

We will need some ideas from the theory of foliations during our analysis of the variation of a family
of varieties.  For completeness, we provide a brief summary of some relevant definitions and results from
the study of foliations on algebraic varieties.

Given a normal variety $X$, we define a foliation $\mathcal F$ to be the datum
of a saturated subsheaf $T_{\mathcal F} \subset T_X$, called the tangent sheaf of the foliation,
such that $T_{\mathcal F}$ is closed under Lie bracket.
The rank of $\mathcal F$ is defined to be the rank of $T_{\mathcal F}$
as a sheaf.
We define the normal sheaf to be $N_{\mathcal F} \coloneqq (T_X/T_{\mathcal F})^{**}$ and the cotangent sheaf to
be $\Omega^1_{\mathcal F} \coloneqq T_{\mathcal F}^*$.

We refer to \cite{Brunella00} for basic notions and definitions regarding foliations.
We recall that given a smooth foliation $\mathcal F$ on a smooth variety $X$, at any point $x \in X$ 
there is an analytic neighbourhood $U$ and a holomorphic submersion $F\colon U \rightarrow V$
such that $T_{\mathcal F}|_U = T_{U/V}$.  The leaves of $\mathcal F$ are locally given by the fibres of this submersion.

We refer to \cite[\S~2.3]{CS21} for the notion of the transform of a foliation under a birational map.

Given a variety, a foliation $\mathcal F$ and a sheaf $E$, we say that $E$ admits a partial $\mathcal F$-connection
provided there exists a $\mathbb C$-linear map 
$\nabla\colon E \rightarrow E \otimes \Omega^1_{\mathcal F}$ which satisfies the
Leibniz rule.
We recall that the normal sheaf of a foliation $N_{\mathcal F}$ is always equipped with a partial $\mathcal F$-connection, called
the Bott connection.

\begin{lemma}
\label{lem_partial_connection_invariant}
Let $X$ be a normal variety, let $\mathcal F$ be a foliation on $X$ and let $L$ be a line bundle 
equipped with a partial $\mathcal F$-connection, 
$\nabla\colon L \rightarrow L \otimes \Omega^1_{\mathcal F}$.
Let $s \in H^0(X, L)$ and suppose that $\nabla(s) = 0$.  Then $D = (s = 0)$ is $\mathcal F$-invariant.
\end{lemma}
\begin{proof}
The invariance of a divisor may be checked on a big open subset and so 
we may freely assume that $X$
and $\mathcal F$ are smooth.

Let $\{U_i\}$ be a trivialising open cover for $L$ and let $t_i$ be local generators for $L$.
Write $s = f_it_i$, where $f_i = 0$ is a local equation for $D$.  We may compute  
\[
0 = \nabla(s) = \nabla(f_it_i) = df_i\otimes t_i + f_i\nabla(t_i) \mod L \otimes  N^*_{\mathcal F}.
\]
Since $\nabla(t_i) = \omega\otimes t_i$ for a holomorphic one form $\omega$, we
see that $df_i +f_i\omega = 0 \mod N^*_{\mathcal F}$,
from which we may conclude that after restricting to $D$ we have  $df_i\vert_D = 0 \mod N^*_{\mathcal F}\vert_D$.
This implies that $\{f_i = 0\}$ is $\mathcal F$-invariant
as required. 
\end{proof}

Let $X$ be a smooth variety, let $X_0 \subset X$ be a local complete intersection (l.c.i.)
subvariety and let $\mathcal F$ be a smooth foliation
on $X$.  We say that $X_0$ is everywhere transverse to $\mathcal F$ provided the natural map 
$T_{\mathcal F}\vert_{X_0} \rightarrow N_{X_0/X}$ is an isomorphism.  

\begin{lemma}
\label{lem_transverse_to_locally trivial}
Let $f\colon P \rightarrow Z$ be a proper fibration between complex manifolds, let $\mathcal F$
be a smooth foliation on $P$ with $\rk \mathcal F = \dim Z$ 
and let $Y$ be a $\mathcal F$-invariant subvariety.  
Let $z \in Z$ be a closed point and let $P_z$ and $Y_z$ be the fibre
over $z$ of $P$ and $Y$ respectively.
  
If $P_z$ is everywhere transverse to $\mathcal F$ (in particular $P_z$ is l.c.i.) then there exists an open Euclidean neighbourhood
$U\subset Z$ of $z$
such that $(f^{-1}(U), Y \cap f^{-1}(U)) \cong (P_z, Y_z) \times U$ as analytic spaces and 
$\mathcal F\vert_{f^{-1}(U)}$ is given by projection onto the first coordinate.
\end{lemma}
\begin{proof}
Since $P_z$ is everywhere transverse to $\mathcal F$, by definition, the natural map $T_P\vert_{P_z} \rightarrow N_{P_z/P}$
is surjective, in particular, $P_z$ is smooth.

We recall the graphic neighbourhood construction of \cite[\S~2.1]{bm16}.  
We remark that \cite{bm16} only considers
the graphical neighbourhood in the case where $P_z$ is a curve, 
but the construction works for general smooth
varieties just as well.
Let $\delta \colon P_z \rightarrow P_z \times P$
be the diagonal embedding.  We may find a germ $\Gamma$ of an analytic variety containing $\delta(P_z)$
so that for all $t \in P_z$ we have $q(p^{-1}(t)) \subset P$ is a germ of the leaf of $\mathcal F$ through $t$,
where $p\colon \Gamma \rightarrow P_z$ and $q\colon \Gamma \rightarrow P$ are the projections.

Observe $\dim \Gamma = \dim P$
and so in fact $q$ gives an isomorphism between an open neighbourhood $V$ of $P_z$ and $\Gamma$.
By the properness of the fibration, perhaps shrinking this neighbourhood, 
we may assume that $V = f^{-1}(U)$ for some open 
Euclidean neighbourhood of $z$
and so
we get a morphism $(p\circ q^{-1}, f)\colon V \rightarrow P_z \times U$ which is easily seen to be an isomorphism.
 
Since the fibres of $p\circ q^{-1}$ are exactly 
the leaves of $\mathcal F$ we see that if $Y$ is an invariant subvariety, then $(p\circ q^{-1})(Y) = Y_z$,
and we may conclude.
\end{proof}

We also recall some well known properties relating holomorphic connections and foliations.

\begin{lemma}
\label{lem_flatconnsummary}
Let $X$ be a smooth variety and let $E$ be a vector bundle on $X$ which admits a flat holomorphic connection
$\nabla\colon E \rightarrow E \otimes \Omega^1_X$.
Let $p\colon P \coloneqq  \mathbb P(E) \rightarrow X$ be the associated
projective space bundle.
Then
\begin{enumerate}
\item $\nabla$ gives a splitting $p^*T_X \rightarrow T_P$ whose image defines a smooth foliation $\mathcal F$ on $P$;

\item $\mathcal O_P(1)$ admits a partial $\mathcal F$-connection; and

\item if $E'$ is a vector bundle on $P$ and admits a partial $\mathcal F$-connection, then $p_*E'$
is locally free and admits a holomorphic connection.
\end{enumerate}

\end{lemma}

\begin{proof}
Item (1) is a standard fact and we omit the proof.  

\medskip
To prove item (2) let $U \subset X$ be a small open subset such that $T_X$ is generated by
$\partial/\partial x_1, \dots, \partial/\partial x_n$ 
and let $s_i$ be a basis of flat sections of $E$. 
Let $\partial_i$ be the lift of $\partial/\partial x_i$ to a vector field on $P$ given by $\nabla$.

Write $p^{-1}(U) = \cup_i U_i$ where $U_i = \{\overline{s}_i \neq 0\}$, where we denote
by $\overline{s}_i$ the section of
$\mathcal O_P(1)$ corresponding to $s_i$. 
 
We define a partial connection $\overline{\nabla}$ on $\mathcal O_P(1)$
by requiring $\overline{\nabla}(\overline{s}_i) = 0$ for all $i$.  This is well defined
because $dg_{ij} = 0 \mod N^*_{\mathcal F}$ where 
$g_{ij} = \overline{s_i}/\overline{s_j}$.
\medskip

The existence of a connection on $p_*E$ is a straightforward consequence of the fact that $p_*\Omega^1_{\mathcal F} = \Omega^1_X$.  The fact that $p_*E$ is locally free follows from
\cite[Corollaire 2.5.2.2]{MR1862024}.
\end{proof}

We remark that under the correspondence in item (1) of 
Lemma \ref{lem_flatconnsummary} the leaves of $\mathcal F$ are exactly the flat
sections of $\nabla$.

Given a normal variety $X$ we say a resolution of singularities $p\colon X' \rightarrow X$ is equivariant
provided $p_*T_{X'} = T_X$.  Such resolutions exist by \cite[Theorem 3.26]{MR2289519}.
We remark that if $\mathcal F$ is a foliation on $X$ such that $T_{\mathcal F}$ is locally free, 
then there exists a foliation on $X'$ such that $p_*T_{\mathcal F'} = T_{\mathcal F}$.

\subsection{Numerical flatness to isotriviality}

We recall that a vector bundle $E$ on a smooth variety is said to be numerically flat provided 
$E$ is nef and $-c_1(E)$ is nef.

\begin{lemma}
\label{lem_flat_to_isotrivial}
Let $f\colon (X, B) \rightarrow Z$ be a projective fibration over a 
smooth projective curve $Z$, where $B \geq 0$, $(X, B)$
is klt, and $K_X+B$ is $f$-semi-ample.

Suppose that for all $m \gg 0$ sufficiently divisible 
we have that the sheaf $E_m \coloneqq f_*\mathcal O_X(ml(K_{X/Z}+B))$ is locally free
and numerically flat.

Let $h\colon X \rightarrow Y$ be the relatively ample
model of $(X, B)$ over $Z$, and let $(Y,B_Y,\bM.)$ be the generalised pair induced by the LC trivial fibration $h$.

\begin{enumerate}
\item There exists a foliation $\mathcal F_Y$ on $Y$ everywhere transverse to the fibres of 
$Y \rightarrow Z$, such that if 
$0 \leq D$ is any divisor such that 
$D \sim_{\mathbb Q}\bM Y.$, then ${\rm Supp}(B_Y+D)$ is $\mathcal F_Y$ invariant.

\item In particular, $(Y, B_Y+D) \rightarrow Z$ is a locally trivial family.

\item If $(Y_0, (B_Y+D)_0)$ denotes a general
fibre of $(Y, B_Y+D) \rightarrow Z$, then there exists a finite \'etale cover $Z' \rightarrow Z$
such that $(Y, B_Y+D)\times_ZZ' \cong (Y_0, (B_Y+D)_0)\times Z'$.

\end{enumerate}
\end{lemma}
\begin{proof}

{\bf Step 1.} In this step we construct a foliation. 

For $m \gg 0$ and sufficiently divisible, 
we have an embedding $Y \rightarrow P  \xrightarrow{p} Z$, where $P\coloneqq \mathbb P(E_m)$.  
Set $L = \mathcal O_P(1) \vert_Y$.
The connection $\nabla_m$ on $E_m$ guaranteed to exist
by \cite[\S~3]{Simpson92} is in fact a holomorphic connection 
(see \cite{MR4376097} or  \cite[Theorem 3.5]{Ou21})
and therefore by Lemma \ref{lem_flatconnsummary} gives a smooth foliation $\mathcal H$ on $P$
such that $T_{\mathcal H} \cong p^*T_Z$, in particular $T_{\mathcal H}$ is locally free.
Observe that $K_{\mathcal H} \cong p^*K_Z$.
Moreover, $\mathcal H$ is everywhere transverse to 
the fibres of $P \rightarrow Z$.

It follows from \cite[Corollary 3.10]{Ou21} that for all $k>0$ the kernel of 
$\Sym^kE_m \rightarrow E_{mk} \rightarrow 0$
is a
$\Sym^k\nabla_m$-invariant subbundle, and therefore  
$Y$ is invariant by $\mathcal H$.
We may then restrict $\mathcal H$ to $Y$ to get a foliation $\mathcal F_Y$ on $Y$.

	{\bf Step 2. } In this step we show local triviality of $(Y, B_Y+D) \rightarrow Z$. 

By considering $Y$ as a subscheme of $P$ which is invariant under $\mathcal H$ and applying Lemma 
\ref{lem_transverse_to_locally trivial} we see that $Y \rightarrow Z$ is locally trivial.
It remains to show that $(Y, B_Y+D)$ is locally trivial.  This will follow by another application 
of Lemma \ref{lem_transverse_to_locally trivial} by considering ${\rm Supp}(B_Y+D)$ as a subscheme of $P$
if we can show ${\rm Supp}(B_Y+D)$ is $\mathcal F_Y$-invariant (and therefore $\mathcal H$ invariant as a subscheme of $P$).

We may find $\tilde{D} \sim_{\mathbb Q} D$ such that ${\rm Supp}(D) \subset {\rm Supp}(\tilde{D})$ and such that 
$(Y, B_Y+\tilde{D})$ is klt.   It suffices to prove all the claims of the Lemma for $(Y, B_Y+\tilde{D})$, so 
without loss of generality we may freely replace $D$ by $\tilde{D}$ and so may assume that $(Y, B_Y+D)$ is klt.
Let $s\colon (Y', B_Y'+D') \rightarrow (Y, B_Y+D)$ be a 
small $\mathbb Q$-factorialisation, which exists by \cite[Corollary 1.4.4]{BCHM06},
and let 
$\sigma\colon (\overline{Y}, \overline{B}_Y+\overline{D}) \rightarrow (Y, B_Y+D)$
be the ample model of $K_{Y'}+(1+\eta)(B'_Y+D')$ over $Y$ for $0<\eta \ll 1$, which exists 
by \cite[Theorem 1.2]{BCHM06}.
Note that $\overline{B}_Y+\overline{D}$ is $\sigma$-ample by construction.

Set $\overline{\mathcal F}_Y = \sigma^{-1}\mathcal F_Y$ and 
note that since $\sigma$ is small, $K_{\overline{\mathcal F}_Y} = \sigma^*K_{\mathcal F_Y}$.
By Lemma~\ref{lem_flatconnsummary} 
$\mathcal O_{\overline{Y}}(m(K_{\overline Y/Z}+\overline{B}_Y+\overline{D}))$ admits a partial 
$\overline{\mathcal F}_Y$-connection.

By the Bott partial connection we know that 
$\det N^*_{\overline{\mathcal F}_Y} \cong K_{\overline{Y}} - 
K_{\overline{\mathcal F}_Y} = K_{\overline{Y}/Z}$ 
also admits a partial connection.
It follows that $\mathcal O_{\overline{Y}}(m(\overline{B}_Y+\overline{D}))$ admits a partial $\overline{\mathcal F}_Y$-connection,
and so 
$\mathcal O_{\overline{Y}}(m(K_{\overline Y/Z}+\overline{B}_Y+D)+k(\overline{B}_Y+\overline{D}))$ 
admits a partial
$\overline{\mathcal F}_Y$-connection for all $m \gg k \gg 0$.

By Lemma \ref{lem_flatconnsummary}, 
$\tilde{E}_m \coloneqq p_*\mathcal O_{\overline{Y}}(m(K_{\overline Y/Z}+\overline{B}_Y+\overline{D})+k(\overline{B}_Y+\overline{D}))$
is a vector bundle which admits a holomorphic connection.  It follows that $\det \tilde{E}_m$ admits a holomorphic connection and so $c_1(\tilde{E}_m) = 0$.
By \cite[Corollary 4.5]{MR3294622} (see also \cite[Theorem 1.11]{MR3779955})  $\tilde{E}_m$ is nef. Since $\tilde{E}_m$ is nef and $-c_1( \tilde{E}_m)$ is nef we conclude that $\tilde{E}_m$ numerically flat.  
Let $\tilde{\nabla}_m$ be the connection guaranteed to exist by  
\cite[\S~3]{Simpson92}.  Let $s \in H^0(\overline{Y}, \mathcal O_{\overline{Y}}(k(\overline{B}_Y+\overline{D}))$ be a 
global section such that $\{s = 0\} = k(\overline{B}_Y+\overline{D})$, and let $\varphi\colon E_m \rightarrow \tilde{E}_m$
be the morphism induced by multiplication by $s$.  By the proof of \cite[Corollary 3.10]{Ou21}
we see that $\varphi$ is a morphism of vector bundles equipped with connections, i.e., if $t$ is any flat local section
of $E_m$, we see that $t\otimes s$ is a flat local section of $\tilde{E}_m$.

Without loss of generality 
we may assume that $m \gg k \gg 0$ are sufficiently large 
so that $m(K_{\overline Y/Z}+\overline{B}_Y+\overline{D})+k(\overline{B}_Y+\overline{D})$
is $\bar{p}$-very ample.  Let $\tilde{P} \coloneqq \mathbb P(\tilde{E}_m)$ and by repeating the 
construction of Step 1, we produce a foliation $\tilde{\mathcal F}$ on $\overline{Y}$ which is
everywhere transverse to the fibres of $\overline{Y} \rightarrow Z$.  
Let us continue to denote by $\tilde{\nabla}_m$ the $\tilde{\mathcal F}$-partial connection on 
$m(K_{\overline Y/Z}+\overline{B}_Y+\overline{D})+k(\overline{B}_Y+\overline{D})$ guaranteed by Lemma 
\ref{lem_flatconnsummary}.  As noted, if $t$ is any flat local
section of $E_m$, then $\tilde{\nabla}_m(t\otimes s) = 0$, and so by 
Lemma \ref{lem_partial_connection_invariant} it follows that $\{t\otimes s = 0\}$ is invariant, from 
which we may conclude that ${\rm Supp}(\overline{B}_Y+\overline{D}) = {\rm Supp}(\{s = 0\})$ is invariant.

Since $\varphi$ was a morphism of vector bundles equipped with connections it follows that 
$\mathcal F_Y = \sigma_*\tilde{\mathcal F}$, and so $\mathcal F_Y$ leaves $B_Y+D$ invariant. 
\medskip

Our final claim follows by first observing that $(Y_0, (B_Y+D)_0)$
is klt.  We then recall that the automorphism group of a log canonical pair of general type
is finite, and so our locally trivial family is in fact trivial after a finite \'etale base change.
\end{proof}

We remark that the use of \cite{Simpson92} in the above proof is closely related to a similar application in 
\cite{MR3959071}.  The use of \cite{Simpson92} to produce a foliation transverse to a fibration
was also considered in \cite{Ou21}.

\begin{proposition}
\label{prop_lift}
Let $f\colon (X, B) \rightarrow Z$ be a projective fibration over a 
smooth projective curve $Z$, where $B \geq 0$,  $(X, B)$
is klt and $K_X+B$ is $f$-semi-ample.

Suppose that for all $m \gg 0$ sufficiently divisible 
$E_m \coloneqq f_*\mathcal O_X(ml(K_{X/Z}+B))$ is locally free
and numerically flat.

Let notation be as in Lemma \ref{lem_flat_to_isotrivial}.
\begin{enumerate}
\item There exists a resolution of singularities $(\overline{X}, \overline{B}) \rightarrow (X, B)$ such that 
the foliation $\mathcal F_Y$ lifts to a foliation $\mathcal F_{\overline{X}}$ on $\overline{X}$.
Moreover, there is a Zariski open subset $U \subset Y$ such that for $y \in U$ 
\[T_{\mathcal F_{\overline{X}}}\vert_{\overline{X}_y} \to N_{\overline{X}_y/\overline{X}}\] is a subbundle
and $\mathcal F_{\overline{X}}$ 
leaves $\overline{B} \cap V$ 
invariant.  

\item For general $t, t' \in Z$, 
there exists an isomorphism $\psi\colon Y_t \rightarrow Y_{t'}$ 
such that if $(\overline{X}_s, \overline{B}_s)$ denotes
the fibre of $(\overline{X}, \overline{B}) \rightarrow Y$ over $s$, then for general $s \in Y_t$, 
$(\overline{X}_s, \overline{B}_s)
\cong (\overline{X}_{\psi(s)}, \overline{B}_{\psi(s)})$.

\item For general $t, t' \in Z$ the geometric generic fibres of $(X_t, B_t) \rightarrow Y_t \cong Y_0$ and  
$(X_{t'}, B_{t'}) \rightarrow Y_{t'} \cong Y_0$  are isomorphic, in particular, 
there exists a finite cover $\tilde{Y}_0 \rightarrow Y_0$ 
(perhaps depending on $t$ and $t'$) such that  
$(X_t, B_t)\times_{Y_0}\tilde{Y}_0$ and $(X_{t'}, B_{t'})\times_{Y_0}\tilde{Y}_0$ are birational.

\item If $y \in Y$ is a general point and $L$ is the leaf of $\mathcal F_Y$ passing through $y$, 
then the family $(X, B) \times_Y L \rightarrow L$ is isotrivial.

\end{enumerate}
\end{proposition}
\begin{proof}
Following \cite[\S~2]{Ambro05}, we may find a generically finite morphism 
$\sigma\colon \widetilde{Y} \rightarrow Y$, a morphism $\tau\colon \widetilde{Y} \rightarrow Y^!$ and a big and nef divisor
$\bM Y^!.$ such that $\overline{\tau^*\bM Y^!.} = \sigma^* \bM .$ as equality of b-divisors.
Notice that we use b-divisors to express this identity, since $\bM.$ may not descend onto $Y$, 
and so $\bM.$ may not agree with the Cartier closure of $\bM Y.$.
It follows that for some $k \gg 0$ some sublinear system of 
$~|k \bM Y.|~$ 
defines
a rational map $\phi\colon Y \dashrightarrow W$ such that, at the generic point $\eta$ of $Y$,
	$~\ker {\rm d}\phi_{\eta}$ agrees with $\ker \kappa$. Here ${\rm d}\phi_{\eta}$ is the differential 
	$T_{Y, \eta} \rightarrow (f^*T_W)_\eta$ restricted to the generic fibre, and 
$\kappa\colon T_{Y, \eta} \rightarrow H^1(T_{\overline{X}_\eta}(-\log \overline{B}))$ is the Kodaira--Spencer map, see \cite[Theorem 2.2]{Ambro05}, and 
$(\overline{X}, \overline{B}) \rightarrow (X, B)$
is a log smooth model as guaranteed by \cite[Lemma 1.1]{Ambro05}.
Observe that we may not take the full linear series $~|k \bM Y.|~$, since $\bM.$ may not descend onto $Y$.
Let $\bar{h}\colon \overline{X} \rightarrow Y$
be the natural morphism.

The general fibre of $Y \dashrightarrow W$ is generically an intersection of divisors 
$D \sim k \bM Y.$.
By Lemma \ref{lem_flat_to_isotrivial}, these divisors are all $\mathcal F_Y$
invariant, and so the general fibre is $\mathcal F_Y$ invariant.It follows that at the generic
point of $Y$, $T_{\mathcal F_Y}$ is 
contained in the kernel of the Kodaira--Spencer morphism
$T_{Y} \rightarrow R^1\bar{h}_*T_{\overline{X}/Y}(-\log \overline{B})$,
	as the latter generically coincides with $\ker {\rm d}\phi$.
Thus, at the generic point of $Y$ we have a lift
$\rho\colon h^*T_{\mathcal F_Y} \rightarrow T_{\overline{X}}(-\log \overline{B})$, 
which induces a foliation $\mathcal F_{\overline{X}}$ on $\overline{X}$. Let $U$ be the largest open subset
over which we have a lift $h^*T_{\mathcal F_Y} \rightarrow T_{\overline{X}}(-\log \overline{B})$. 

By construction $T_{\mathcal F_{\overline{X}}}\vert_{\overline{X}_y} \to N_{\overline{X}_y/\overline{X}}$ is a subbundle where $\overline{X}_y$ is a fibre of $\overline{X} \rightarrow Y$ over $U$. Since $\rho$ takes values in 
$T_{\overline{X}}(-\log \overline{B})$ we see that $\mathcal F_{\overline{X}}$ leaves
the support of $\overline{B}$ invariant.
This proves item (1).

We now prove item (2).
Fix a general point $t \in Z$,
and let $t' \in Z$ be a nearby point.  
By Lemma~\ref{lem_flat_to_isotrivial}, we have isomorphisms $\psi\colon Y_t \rightarrow  Y_{t'}$ 
given by the flows of the foliation $\mathcal F_Y$. 
Perhaps shrinking $U$, we may assume this gives an isomorphism $\psi\colon U_t \rightarrow U_{t'}$.

For any closed point $s \in U_t$, let $L$ denote a germ of a leaf of $\mathcal F_Y$ through $s$, 
which up to shrinking, we may assume is contained in $U$.
By construction, we have $L \cap U_{t'} = \psi(s)$.
Consider the family $(\overline{X}_L, \overline{B}_L) \coloneqq (\overline{X}, \overline{B}) \times_Y L$.  
By construction, $\overline{X}_L$ is $\mathcal F_{\overline{X}}$-invariant, and so we have a restricted foliation 
$\mathcal F_{\overline{X}_L}$ which
is everywhere transverse to the fibres of $\overline{X} \rightarrow L$ and leaves $B$ invariant.  
Applying Lemma~\ref{lem_transverse_to_locally trivial} to $\mathcal F_{\overline{X}_L}$ 
and $(\overline{X}_L, \overline{B}_L) \rightarrow L$, 
it follows that $(\overline{X}_s, \overline{B}_s) \cong (\overline{X}_{\psi(s)}, \overline{B}_{\psi(s)})$, as required.

By considering the relative ${\rm Isom}$ functor ${\rm Isom}_{Y_0}((X_{t}, B_{t}), (X_{t'}, B_{t'}))$
for general $t, t'$,
we see that standard arguments imply 
the geometric generic fibres $(X_t, B_t) \rightarrow Y_0$  and $(X_{t'}, B_{t'}) \rightarrow Y_0$  are isomorphic, from which we may conclude 
item (3).

Item (4) is a direct consequence of the previous arguments.  We remark that by the last point in  
Lemma \ref{lem_flat_to_isotrivial} the leaves of $\mathcal F_Y$ are all algebraic.
\end{proof}

\subsection{Harder--Narasimhan filtrations}

\label{s_HN}

Given a vector bundle $E$ on a smooth projective curve we define its slope to be
$\mu(E) \coloneqq \frac{\deg E}{\rk E}$.
We define $\mu_{\max}(E)$ to be the largest slope $\mu(F)$ among all non-zero
 subbundles $0 \rightarrow F \rightarrow E$,
and we define $\mu_{\min}(E)$ to be the smallest slope $\mu(Q)$ among non-zero quotient bundles $E \rightarrow Q \rightarrow 0$.

\begin{definition}
\label{defn_lambda}
Let $f\colon (X, B) \rightarrow Z$ be a projective fibration to a
smooth projective curve, where $B \geq 0$, $(X, B)$
is log canonical.
Let $l$ be a positive integer such that $l(K_{X/Z}+B)$ is Cartier.
Let $P \in Z$ be a closed point.
We define 
\[\lambda_-(X/Z, B) \coloneqq \sup \{ t : l(K_{X/Z}+B)-tf^*P \text{ is nef}\}\]
and 
\[\lambda_+(X/Z, B) \coloneqq \sup \{ t : l(K_{X/Z}+B)-tf^*P \text{ is pseudo-effective}\}.\]

When clear from context we will write $\lambda_- = \lambda_-(X/Z, B)$ and $\lambda_+ = \lambda_+(X/Z, B)$.
\end{definition}

We remark that $\lambda_-$ and $\lambda_+$ depend on the choice of $l$.  We feel this is unlikely to cause
confusion and so do not include the choice of $l$ in our notation.

\begin{proposition}
\label{prop_threshold_to_slope}
Let $f\colon (X, B) \rightarrow Z$ be a GLC fibration where $B \geq 0$,  
$(X, B)$ is log canonical and klt over the generic point of $Z$, 
$Z$ is a projective curve and $B \geq 0$.
Let $l>0$ be a sufficiently divisible integer such that $L \coloneqq l(K_{X/Z}+B)$ is Cartier
and suppose that for all $m \gg 0$ we have that $E_m \coloneqq f_*\mathcal O_X(ml(K_{X/Z}+B))$ is locally free.
Suppose in addition that $K_{X/Z}+B$ is relatively
semi-ample.
Then, for $m \gg 0$, we have
\[
\lambda_- \leq \frac{1}{m}\mu_{{\rm min}}(E_m) \leq \frac{1}{m}\mu_{{\rm max}}(E_m) \leq \lambda_+.
\]
\end{proposition}
\begin{proof}
Suppose first that $K_{X/Z}+B$ is relatively ample.  In this case the result 
follows from \cite[Lemma-Definition 2.26]{XZ20} (see also 
\cite[Proposition 5.4]{CoPa21} and \cite[Proposition 6.4]{CoPa21}) by taking $L = l(K_{X/Z}+B)$.

In general, let $g\colon X \rightarrow Y$ be the ample model of $K_{X/Z}+B$ over $Z$, and, perhaps replacing $l$ by a multiple 
we may find a Cartier divisor $L'$ be such that $g^*L' \sim L$.
By \cite[Theorem 0.2]{Ambro05}, we may write $L' \sim_{\mathbb Q} K_{Y/Z}+B_Y$ 
where $(Y, B_Y)$ is klt over the generic point of $Z$.
We have equalities $\lambda_+(X/Z, B) = \lambda_+(Y/Z, B_Y)$ and $\lambda_-(X/Z, B) = \lambda_-(Y/Z, B_Y)$
and so we may freely reduce to the previous case
in order to conclude.
\end{proof}


\section{Property $(*)$ pairs and locally stable families}

\begin{lemma}
\label{lem_mod_part_stable}
Let $f\colon (X, B) \rightarrow Z$ be a GLC contraction which is also a 
locally stable family of log varieties.
The moduli part of $(X/Z, B)$
is $K_{X/Z}+B$. 
\end{lemma}
\begin{proof}
By Proposition \ref{prop_kollar_criterion}
it follows that the discriminant $B_{Z}$ of $(X/Z, B)$ 
is zero, and we may conclude.
\end{proof}

The following openness property of Property $(*)$ morphisms is implicit in \cite{ACSS}.

\begin{proposition}
\label{prop_property*_open}
Let $f\colon (X, B) \rightarrow Z$ be a GLC contraction where $B \geq 0$. 
Suppose that there exists a reduced divisor $\Sigma \subset Z$ such that 
$(Z, \Sigma)$ is log smooth and such that the vertical part of $B$ coincides with 
$f^{-1}(\Sigma)$.
Suppose moreover for all closed points $z \in Z$ there is a 
reduced divisor $D \subset Z$ such that $(Z, \Sigma+D)$ is log smooth, 
$z$ is a log canonical centre of $(Z, \Sigma+D)$, and $(X, B+f^*D)$ 
is log canonical in a neighbourhood of $f^{-1}(z)$.
Then $(X/Z, B)$ satisfies Property $(*)$.
\end{proposition}

\begin{proof}
%
The proof of \cite[Proposition 2.19]{ACSS} implies the following fact.
Let $z' \in Z$ be a point and $D_{z'}$ a
divisor such that 
\begin{enumerate}
\item $(Z, \Sigma+D_{z'})$ is log smooth;
\item $z'$ is a log canonical centre of $(Z, \Sigma+D_{z'})$; and
\item  $(X, B+f^*D_{z'})$ is log canonical.
\end{enumerate}
Then, if $D'_{z'}$ is any other divisor such 
that $(Z, \Sigma+D'_{z'})$ is log smooth and $z'$ 
is a log canonical centre of $(Z, \Sigma+D'_{z'})$, it follows 
that $(X, B+f^*D'_{z'})$ is log canonical.
This in turn immediately implies the proposition.
\end{proof}

\begin{lemma}
\label{lem_dlt_mod_*}
Let $f\colon (X, B,\bM.) \rightarrow Z$ be a generalised GLC contraction.
Consider a generalised dlt modification $p\colon (X', B',\bM.) \rightarrow (X, B,\bM.)$.

If $(X/Z, B,\bM.)$ satisfies generalised Property $(*)$, then $(X'/Z, B',\bM.)$ 
satisfies generalised Property $(*)$.
\end{lemma}
\begin{proof}
Let $\Sigma \subset Z$ be the reduced divisor associated to $(X/Z, B,\bM.)$.

Note that $(X, B,\bM.)$ is generalised log canonical.
From the equality $K_{X'}+B'+\bM X'.= p^*(K_X+B+\bM X.)$, it is immediate that $(X, B+f^*D,\bM.)$ is generalised log canonical if and only if $(X', B'+p^*f^*D,\bM.)$ is generalised log canonical.

Thus to conclude it suffices to show that if $E$ is a vertical component of $\lfloor B' \rfloor$, then 
$p(E)$ is contained in a vertical component of $\lfloor B\rfloor$.
Generalised Property $(*)$ implies that, if $W \subset X$ is any generalised log canonical centre of $(X,B,\bM.)$, 
then by Lemma \ref{lem_img_lc_centre_is_lc} $f(W)$ is a log canonical centre of $(Z, \Sigma)$ and 
so $W$ is contained in $\lfloor B \rfloor$.
Since $p(E)$ is a log canonical centre, we may therefore conclude.
\end{proof}

\begin{lemma}
\label{lem_horz_has_*}
Let $f\colon (X, B) \rightarrow Z$ be a GLC contraction satisfying Property $(*)$ with $B \geq 0$.
Suppose in addition that 
\begin{enumerate}
\item $f\colon X \rightarrow Z$ is equidimensional; and 
\item if $D \subset Z$ is any reduced divisor, then $f^*D$ is reduced.
\end{enumerate} 
Then $(X, B^h)$ satisfies Property $(*)$.
\end{lemma}
\begin{proof}
We first claim that $K_X+B^h$ is $\mathbb Q$-Cartier.  Observe that Property $(*)$ implies that, if $D$ is a vertical component of $B$, then in fact $D$ is a component of $\lfloor B\rfloor$.
This together with item (2) implies that $B^v = f^*G$ where $G$ is a reduced smooth normal crossings divisor on $Z$.

For any $z \in Z$, let $G' \geq 0$ be reduced divisor such that $(Z, G+G')$ 
is log smooth and $z$ 
is a log canonical centre of $(Z, G+G')$.
Since $(X, B)$ satisfies Property $(*)$, we see that
$(X, B^h+f^*(G+G'))$ is log canonical.
Since $z \in Z$ is arbitrary, we may apply Proposition \ref{prop_property*_open} to conclude $(X, B^h)$ satisfies Property $(*)$.
\end{proof}

\begin{proposition}
\label{prop_equi*_cover}
Let $f\colon (X, B) \rightarrow Z$ be an equidimensional GLC 
contraction which satisfies Property $(*)$ and where $B \geq 0$.

Then, there exists a diagram

\begin{center}
\begin{tikzcd}
Y \arrow[d, "g"] \arrow[r, "s"] & X \arrow[d, "f"] \\
W \arrow[r, "r"] & Z
\end{tikzcd}
\end{center}

where 
\begin{enumerate}
\item $r\colon W \rightarrow Z$ is a finite Galois morphism; 
\item  $(Y, C)$ is the normalisation of $(X, B^h)\times_ZW$; and
\item $g\colon (Y, C) \rightarrow W$ is a locally stable family.
\end{enumerate}

Moreover, if $M$ (resp. $N$) is the moduli part of $(X/Z, B)$ (resp. $(Y/W, C)$) 
and $M$ is $f$-semi-ample,
then 
\begin{enumerate}
\item $s^*M \sim N$, where $s\colon Y \rightarrow X$ is the natural morphism;
\item   for all $k >0$ sufficiently divisible, $r^*f_*\mathcal O_X(kM)$ is locally free; and 
 \item if, in addition, $f\colon (X, B^h) \to Z$ is a locally stable family, then for all $k >0$ sufficiently divisible, $r^*f_*\mathcal O_X(kM) \cong g_*\mathcal O_Y(kN)$.
\end{enumerate}
\end{proposition}

\begin{proof}
Let $\Sigma_Z \subset Z$ be the reduced divisor associated to the Property $(*)$ pair $(X/Z, B)$.

Because $X \rightarrow Z$ is equidimensional, then away from subsets of codimension $\geq 2$ in $X$ and $Z$ we have
that $f\colon X \rightarrow Z$ is smooth, hence toroidal.
Therefore, by \cite[Proposition 5.1]{AK00}, we may find a 
Kawamata cover $r\colon W\rightarrow Z$ 
which is finite and Galois with Galois group $G$, see \cite[Theorem 1-1-1]{KMM87},
such that if $Y$ is the normalisation of $X\times_WZ$ then for any reduced divisor $D \subset W$ we have
$g^*D$ is reduced, where $g\colon Y \rightarrow W$ is the natural projection.

Let $\Sigma \subset Z$ denote the branch locus of $r \colon W \rar Z$.
If $C \subset Z$ is any divisor such that $f^*C$ is not reduced, then $C \subset \Sigma_Z$ and $C\subset \Sigma$.
Therefore, the prime divisors in $Z$ over which $f$ does not have 
reduced fibres form a simple normal crossing divisor.
By the construction of a Kawamata cover,
we may therefore choose $r\colon W \rightarrow Z$ so that
$\Sigma_Z \subset \Sigma$ and so that $(Z, \Sigma)$ is log smooth.
Replacing $B$ by $B+f^*(\Sigma-\Sigma_Z)$
we may assume that the branch locus of $s\colon Y \rightarrow X$ is contained in $\lfloor B \rfloor$.

Define $\overline{C}$ by the identity $K_Y+\overline{C} = s^*(K_X+B)$.
By the Riemann--Hurwitz formula, we have that $\overline{C} \geq 0$.
We claim that $(Y/W, \overline{C})$ satisfies Property $(*)$.
Indeed, let $\Sigma_W = r^{-1}\Sigma_Z$ and let $w \in W$ be arbitrary.
Let $H$ be a divisor on $Z$ such that $(Z, \Sigma_Z+H)$ is log smooth and 
$r(w)$ is a log canonical centre of $(Z, \Sigma_Z+H)$.
Then $(W, \Sigma_W+r^*H)$ is log smooth and $w$ is a log canonical centre of 
$(W, \Sigma_W+r^*H)$.
By the construction of Kawamata covers, $(W,\Sigma_W)$ is log smooth, 
and the log smoothness of $(W,\Sigma_W+r^*H)$ immediately follows from 
the fact that $\Sigma+H$ is a simple normal crossing divisor.
The claim regarding log canonical centres follows from \cite[Proposition 5.20]{KM98}.
Moreover, $K_Y+\overline{C}+g^*r^*H = s^*(K_X+B+f^*H)$ and so $(Y, \overline{C}+g^*r^*H)$ is log canonical.
Letting $w$ vary over all points of $W$ we may apply 
Proposition \ref{prop_property*_open} to see that $(Y, \overline{C})$ has Property $(*)$. 

By Lemma \ref{lem_horz_has_*}, $(Y, C\coloneqq \overline{C}^h)$ has Property $(*)$ and so
we may apply Lemma \ref{lem_prop*_horz_semi-stable} to see that $(Y, C)$ is locally stable.

To prove (1) we observe have an equality $s^*M = N$ because 
$K_Y+\overline{C} = s^*(K_X+B)$ and $K_W+\Sigma_W = r^*(K_Z+\Sigma_Z)$.
(3) is a consequence of (2) by Grauert's theorem and cohomology and base change \cite[Theorem III.12.8 and Corollary III.12.9]{Hartshorne77}.
To prove (2) it suffices to show that   $h^0(X_z, \mathcal O(kM)|_{X_z})$ is independent of $z \in Z$. 
To check this it suffices to replace $Z$ by a general curves passing through $z \in Z$ so we may freely assume
that $\dim Z = 1$.  Let $c\colon X \to U$ be the relative ample model of $M$ over $Z$, i.e., there exists an $f$-ample divisor $H$ on $U$ such that $M \sim_{\mathbb Q} c^*H$.  For $k \gg 0$ we see that $f_*\mathcal O_X(kM) = p_*\mathcal O_U(kH)$ where $p\colon U \to Z$ is the obvious morphism.  Since $U\to Z$ is flat (2) is then a consequence of Serre vanishing and semi-continuity of fibre dimension.
\end{proof}

\begin{lemma}
\label{lem_base_change_pushforward}
Let $f\colon (X, B) \rightarrow Z$ and $f'\colon (X', B') \rightarrow Z'$
be birationally equivalent equidimensional GLC fibrations with $B, B' \geq 0$
fitting into the following commutative diagram
\[
\begin{tikzcd}
  (X', B') \arrow[r, "\alpha", dashrightarrow] \arrow[d, "f'"] & (X, B) \arrow[d, "f"] \\
  Z' \arrow[r, "\beta"] &  Z.
 \end{tikzcd}
\]

Suppose moreover that the following hold
\begin{enumerate}
\item $(X, B)$ and $(X', B')$ are crepant birational over the generic point of $Z$;

\item the moduli part $M$ (resp. $M'$) of $(X/Z, B)$ (resp. $(X'/Z', B')$) is $f$-nef (resp. $f'$-nef);

\item $(X/Z, B)$ and $(X'/Z', B')$ satisfy Property $(*)$ and are BP stable;

\item $f_*\mathcal O_X(nM)$ is 
locally free for $n \gg 0$ sufficiently divisible; and

\item $\det f'_*\mathcal O_{X'}(nM')$ is nef for $n \gg 0$ sufficiently divisible.
\end{enumerate}

Then for all $n \gg 0$ sufficiently divisible and $k > 0$ we have
isomorphisms
 $\beta^*\det f_*\mathcal O_X(nM) \cong \det f'_*\mathcal O_{X'}(nM')$.  Moreover, if $p\colon W \rightarrow X$ and $q\colon W \rightarrow X'$
 resolves the rational map $X' \dashrightarrow X$ we have 
$q^*(kM'+(f')^*c_1(\det f'_*\mathcal O_{X'}(nM'))) \sim_{\mathbb Q} p^*(kM+f^*c_1(\det f_*\mathcal O_X(nM)))$.
\end{lemma}

\begin{proof}
For ease of notation set $E = f_*\mathcal O_X(nM)$ and $E' = f'_*\mathcal O_{X'}(nM')$.
Let $q\colon R \rightarrow X$ and $p\colon R \rightarrow X'$ resolve the birational map $\alpha \colon X' \dashrightarrow X$.
Let $g\colon R \rightarrow Z'$ be the composition of $f'$ and $p$.
By \cite[Proposition 2.21]{ACSS}, we know that $q^*M \sim_{\mathbb Q} p^*M'$.

By the projection formula,
it follows that for $n \gg 0$ sufficiently divisible we have $\beta_*E' = E$.
This gives us a morphism 
$\beta^*E \rightarrow \beta^*\beta_*E' \rightarrow E'$, which is an isomorphism away from ${\rm exc}(\beta)$.
Taking determinants we get a non-zero morphism $\beta^*\det E \rightarrow \det E'$, which is an isomorphism
away from ${\rm Exc}(\beta)$ and so we can deduce that 
$\det E' \otimes \beta^*\det E^* \sim_{\mathbb Q} \mathcal O(F)$ where $F \geq 0$ is $\beta$-exceptional.

By assumption $\det E'$ is $\beta$-nef and so the same is true of $F$, and so 
the negativity lemma, \cite[Lemma 3.39]{KM98}, implies that $F = 0$.
We deduce that  $\beta^*\det E \cong \det E'$ and
$q^*(kM'+(f')^*c_1(\det E')) \sim_{\mathbb Q} p^*(kM+f^*c_1(\det E))$,
from which we may conclude.
\end{proof}


\section{On semi-ampleness of the moduli part}

\subsection{Stable pairs}

\begin{lemma}
\label{lem_stable_max}
Let $f\colon (X, B) \rightarrow Z$ be a family of stable log varieties where $Z$ is normal and projective.  Let $M$ be the moduli part
of $(X/Z, B)$.

\begin{enumerate}
\item For $m \gg 0$ sufficiently divisible, $f_*\mathcal O_X(mM)$ is 
a vector bundle whose formation commutes with base change.

\item For $m \gg 0$ sufficiently divisible, 
$\lambda_m\coloneqq \det f_*\mathcal O_X(mM)$ is semi-ample.

\item For $m \gg 0$ sufficiently divisible and all $\epsilon \in \QP$, $M+\epsilon f^*c_1(\lambda_m)$ is semi-ample.

\end{enumerate}
\end{lemma}
\begin{proof}
Part (1) is a direct consequence of \cite[Lemma 7.7]{KP17} (and its proof) and the fact that $M = K_{X/Z}+B$.

To show (2) and (3)  consider the following commutative diagram, which is guaranteed to exist
by \cite[Corollary 6.20]{KP17}:
\begin{equation}
\begin{tikzcd}
  (X_{0}, B_0) \arrow[d, "f_0"] & \arrow[l, "s"]  (\overline{X}, \overline{B}) \arrow[r, "\sigma"] \arrow[d, "\overline{f}"] & (X, B) \arrow[d, "f"] \\
   Z_0 & \arrow[l, "t"] \overline{Z} \arrow[r, "\tau"] &  Z
 \end{tikzcd}
\end{equation}
where all three of our vertical morphisms are families of stable log varieties, $\tau$ is finite, and the moduli map associated to $(X_0, B_0) \rightarrow Z_0$ is finite.

By \cite[Corollary 7.3]{KP17} and its proof for $m \gg 0$ sufficiently divisible, 
$\lambda^0_m\coloneqq \det f_{0*}\mathcal O_{X_0}(mM_0)$ is ample, where $M_0$ is the moduli
part of $(X_0/Z_0, B_0)$.
This implies that $M_0+\epsilon f_0^*c_1(\lambda^0_m)$ is ample 
for all $\epsilon \in \QP$, since $M_0$ is globally nef and relatively ample, 
see \cite[Lemma 7.7 and proof of Theorem 7.1.1]{KP17}.
Next, note that we have $\sigma^*M = s^*M_0$ and by (1) we know $\tau^*\lambda_m = t^*\lambda^0_m$, which implies both (2) and (3).
\end{proof}

\subsection{Intermediate Kodaira dimension}

\begin{lemma} \label{lemma_seminormal}
Let $(X,\Delta)$ be a log canonical pair with $\Delta \geq 0$, and let $f \colon (X,\Delta) \rar Z$ be a 
fibration with $\K X. + \Delta \sim_\QQ 0/Z$.
Let $(Z,B_Z,\bM.)$ be the generalised pair induced on $Z$ by the canonical bundle formula.
Then, any union of irreducible components of $\mathrm{Nklt}(Z,B_Z,\bM.)$ is semi-normal.
\end{lemma}

We refer to \cite[\S~2]{FS20} for the definition of $\mathrm{Nklt}(Z,B_Z,\bM.)$.

\begin{proof}
It follows immediately from \cite[Theorem 1.6]{Liu22} and the fact that Du Bois singularities are semi-normal.
\end{proof}

\begin{lemma}[{\cite[Lemma 2.3]{FI21}}] \label{lemma:connected:fivers:vs:cohomol:connected:fibers}
Let $p \colon X\to Y$ be a proper surjective morphism with connected fibres, 
with $Y$ semi-normal and $X$ reduced.
Then $p_* \mathcal{O} _X = \mathcal{O}_Y$.
\end{lemma}

\begin{proposition} \label{semistable_iitaka}
Let $f \colon (X, B) \rar Z$ be a GLC fibration which is also a locally stable family, where $Z$ is smooth.
Assume that $\K X/Z. + B$ is relatively semi-ample over $Z$.
Let $c \colon X \rar Y/Z$ denote the relative Iitaka fibration of $(X, B)$ over $Z$ and
let $(Y,B_Y,\bM.)$ be the generalised pair induced on $Y$.

Assume that the b-semi-ampleness conjecture holds.

Then, for every $z \in Z$, we may find $0 \leq \Delta_Y \sim_\QQ B_Y+\bM Y.$
such that $g \colon (Y, \Delta_Y) \rar Z$ is a stable family of pairs over a neighbourhood of $z \in Z$.
\end{proposition}

\begin{proof}
By Lemma \ref{lem_prop*_horz_semi-stable}, 
$(X, B)\rightarrow Z$ satisfies Property $(*)$, and so by \cite[Theorem 3.1]{Ambro04}
we see that $(Y, B_Y,\bM.)\rightarrow Z$ satisfies generalised Property $(*)$.  
Observe also that if $D \subset Z$ is any reduced divisor, then $g^*D$ is reduced as well.

Let $p\colon \overline{Y}\rightarrow Y$ be a generalised dlt modification, and
let $(\overline{Y}, \overline{B}_Y,\bM.)$ be the transformed pair. 
By Lemma \ref{lem_dlt_mod_*}, we see that $\bar{g}\colon (\overline{Y}, \overline{B}_Y,\bM.) \rightarrow Z$ 
satsifies generalised Property $(*)$, and so it follows that  
$(\overline{Y}/Z, \overline{B}_Y)$ 
satisfies Property $(*)$, and 
let $\Sigma \subset Z$ be the associated boundary.
Since $(X, B) \rightarrow Z$ is locally stable we see
that $\overline{B}_Y^h = \overline{B}_Y$, 
and so by Lemma \ref{lem_prop*_horz_semi-stable} 
$(\overline{Y}, \overline{B}_Y) \rightarrow Z$ is locally stable.

Since we are assuming the b-semi-ampleness conjecture, 
for a general choice of $0 \leq H \sim_{\mathbb Q} \bM \overline{Y}.$, by Bertini's theorem
(applied on a log resolution where $\bM.$ descends) we may guarantee that 
$(\overline{Y}_z, \overline{B}_{Y, z}+H_z)$
is slc. By \cite[Corollary 4.45]{Kollar21} there is an open neighbourhood of $z$ 
such that $(\overline{Y}/Z, \overline{B}_Y+H)$ is locally stable as required.
\end{proof}

\begin{theorem} \label{thm_iitaka}
Let $f \colon (X,\Delta) \rar Z$ be a GLC fibration which is also a locally stable family, where $Z$ is normal.

Assume that the geometric generic fibres over the generic points of the irreducible components of $Z$ are normal.
Assume that, for every $z \in Z$, $\K X_z. + \Delta_z$ is semi-ample.
Then, we may find $m_0 \in \mathbb N_{>0}$ depending on the family $f$ such that the following holds:
\begin{itemize}
    \item[(1)] for every $k \in \mathbb N$, 
	$h^0(X_z,\mathcal{O}\subs X_z. (km_0( \K X_z. + \Delta_z)))$ is independent of $z \in Z$;
    \item[(2)] for every $k \in \mathbb N$, $f_*\mathcal{O}\subs X. (km_0( \K X/Z. + \Delta))$ 
is a vector bundle, whose formation commutes with base change;
    \item[(3)] the relative Iitaka fibration $c \colon (X,\Delta) \rar Y$ of $f$ is a morphism and computes the ample model of each fibre; and
    \item[(4)] the formation of the relative Iitaka fibration commutes with base change by reduced schemes.
\end{itemize}
Furthermore, $g \colon Y \rar Z$ is flat.
In particular, the scheme theoretic fibres of $g$ are pure-dimensional reduced and semi-normal.
\end{theorem}

\begin{proof}
Item (2) immediately follows from item (1) by Grauert's theorem and cohomology 
and base change, see \cite[Corollary III.12.9 and Theorem III.12.11]{Hartshorne77}.
Similarly, item (4) follows from item (3).
Thus, we are left with showing items (1) and (3).

Let us first suppose that $\dim Z = 1$.
Let $c \colon (X,\Delta) \rar Y$ be the relative Iitaka fibration of $f$, which exists by the assumptions and \cite{HX11}.
Since $X$ is normal, so is $Y$.
Let $(Y/Z,B_Y,\bM.)$ be the generalised pair induced by the canonical bundle formula.
Since $Z$ is a curve, $g \colon Y \rar Z$ is flat, see \cite[Proposition III.9.7]{Hartshorne77}.
By semi-continuity \cite[Theorem III.12.8]{Hartshorne77}, it suffices to show that, for a special closed point $0 \in Z$, we have $Y_0 = \Proj(R(X_0,\Delta_0))$, where we set $R(X_0,\Delta_0)=\oplus_{m \in \mathbb N}\Gamma (X_0,\mathcal{O}\subs X_0.(m(\K X_0. + \Delta_0)))$.
Since $f$ is a locally stable family, $(X,\Delta+X_0)$ is log canonical.
Then, by inversion of adjunction for LC-trivial fibrations \cite[Theorem 3.1]{Ambro04}, the induced generalised pair $(Y/Z,B_Y+g^*0,\bM.)$ is generalised log canonical.
In particular, we have $Y_0=g^*0$ as divisors.
In particular, the scheme theoretic fibre $Y_0$ is $R_0$.
Since $Y$ is normal and $g^*0$ is Cartier, it follows that $\mathcal{O}_Y$ and $\mathcal{O}_Y(-g^*0)$ are $S_2$.
Then, by \cite[Corollary 2.61]{Kollar13}, $Y_0$ is $S_1$.
But then, by \cite[Tag 0344]{stacks-project}, $Y_0$ is reduced.
Furthermore, by Lemma \ref{lemma_seminormal}, $Y_0$ is semi-normal.

Now, we have a tower of morphisms as follows
\[
X_0 \xrightarrow{\alpha} \Proj(R(X_0,\Delta_0)) \xrightarrow{\beta} Y_0.
\]
For $m_0$ sufficiently large and divisible, we have the following:
\begin{itemize}
    \item the morphism $\alpha$ is given by the full linear series $|\Gamma (X_0,\mathcal{O}\subs X_0.(m_0(\K X_0. + \Delta_0)))|$; and
    \item the morphism $c_0 \colon X_0 \rar Y_0$ is given by the restricted linear series corresponding to
    \begin{equation} \label{eq_restriction}
    \im  (\Gamma (X,\mathcal{O}\subs X.(m_0(\K X/Z. + \Delta))) \rar \Gamma (X_0,\mathcal{O}\subs X_0.(m(\K X_0. + \Delta_0)))) .
    \end{equation}
    \end{itemize}
Since $c_0$ is given by a sub-series of the one giving $\alpha$, and since $c_0$ contracts exactly the same curves
as $\alpha$, the induced morphism $\beta$ is finite.
By construction, we have $\alpha_* \mathcal{O}\subs X_0. = \mathcal{O}_{\Proj(R(X_0,\Delta_0))}$.
Similarly, we have $c_* \mathcal{O}_X = \mathcal{O}_Y$ and so $X_0 \rar Y_0$ has connected fibres.
Then, by Lemma \ref{lemma:connected:fivers:vs:cohomol:connected:fibers}, we have $(c_0)_*\mathcal{O}_{X_0}=\mathcal{O}_{Y_0}$.
In turn, this gives $\beta_* \mathcal{O}_{\Proj(R(X_0,\Delta_0))}=\mathcal{O}_{Y_0}$.
But then, as $\beta$ is finite, we have $\Proj(R(X_0,\Delta_0))=Y_0$.
Since $g$ is flat, its fibres are pure-dimensional.
This settles items (3) when $\dim Z = 1$.

A priori, \eqref{eq_restriction} may not be an isomorphism, since a proper linear sub-series may still define an isomorphism.
To rule this case out, we observe that we just showed $(c_0)_* \mathcal{O}_{X_0}=\mathcal{O}_{Y_0}$.
Therefore, by the projection formula, in order to prove (1), it suffices to show that $h^0(Y_z,L|\subs Y_z. )$ is independent of $z \in Z$, where $L=\mathcal{O}_Y(m(K_Y+B_Y+\mathbf{M}_Y))$ and $m$ is sufficiently divisible.
But this latter fact readily follows by the flatness of $Y \to Z$ and relative Serre's vanishing applied to the relatively ample line bundle $L$.
This settles (1) when $\dim Z = 1$.

We now handle the general case.
Let $l_0$ denote the Cartier index of $\K X/Z. + \Delta$.
By generic flatness applied to $g \colon Y \rar Z$ and relative Serre's vanishing, we may find a non-empty 
open subset $U \subset Z$ and $m_0 \in \mathbb N$ divisible by $l_0$ such that (1) holds over $U$.
Then, for every closed point $z' \in Z$, by restricting to a general curve passing 
through $z'$ and which meets $U$, 
we may find $m_{z'} \in \mathbb N$ divisible by $m_0$ such that, 
for every $k \in \mathbb N$, $h^0(X_{z'},\mathcal{O}\subs X_z. (km_{z'}( \K X_z'. + \Delta_{z'})))$ 
attains the generic value achieved over $U$.
Then, by semi-continuity and Grauert's theorem \cite[Theorem III.12.8 and Corollary III.12.9]{Hartshorne77}, 
$z'$ is in the open subset where $f_*\mathcal{O}\subs X. (km_{z'}( \K X. + \Delta))$ is a vector bundle, 
for every $k \in \mathbb N$.
Thus, by replacing $m_0$ with a multiple, we may assume that any fixed closed point $z'$ belongs to the open set $U$.
Since the $\mathcal{O}_Z$-algebra $\oplus_{k \in \mathbb N}f_*\mathcal{O}_X(kl_0(\K X/Z.+\Delta))$ 
is finitely generated, after finitely many iterations, the above 
procedure extends the open subset $U$ to the whole $Z$.

Since (1) and (2) are settled and (3) is known over a curve, item (3) now follows immediately.
Similarly, the additional properties of the fibres of $g$ from (3) hold since they 
may be checked in the case the base is a curve.
Lastly, the flatness of $g$ follows from \cite[Theorem 3.20]{Kollar21}.
\end{proof}

\begin{theorem}
\label{thm_locally_stable_semi-ample}
Let $f \colon (X,\Delta) \rar Z$ be a GLC fibration which is also locally stable family, where $Z$ is quasi-projective and normal.
Assume $\K X/Z. +\Delta$ is $f$-nef and 
that a general fibre is a good minimal model.
Further assume that the b-semi-ampleness conjecture holds.

Then, for $m \in \mathbb N_{>0}$ sufficiently divisible,
\begin{enumerate}
\item $\lambda_m \coloneqq \det f_* \mathcal{O}_X(m(\K X/Z. + \Delta))$ is basepoint-free; and

\item for all $\epsilon \in \QP$, $\K X/Z. + \Delta + \epsilon f^* c_1(\lambda_m)$ is semi-ample.
\end{enumerate}
\end{theorem}

\begin{proof}
Since $(X, \Delta) \rightarrow Z$ is locally stable, we know that every log canonical centre of $(X, \Delta)$ dominates
$Z$.
By applying \cite[Theorem 1.1]{HX11} to a dlt modification of $(X, \Delta)$, it follows
that $\K X/Z. + \Delta$ is relatively semi-ample over $Z$.

By Theorem \ref{thm_iitaka}, the coherent sheaf $f_* \mathcal{O}_X(m(\K X/Z. + \Delta))$ is a vector bundle and its formation commutes with base change.
In particular, it follows that the line bundle $\lambda_m\coloneqq \det f_* \mathcal{O}_X(m(\K X/Z. + \Delta))$ is compatible with base change.
Lastly, we recall that, given a line bundle $\mathcal{L}$ on a normal variety $U$ and a surjective projective morphism 
$\phi \colon V \rar U$ where $V$ is normal, then $\mathcal{L}$ is semi-ample if an only if so is $\phi^* \mathcal{L}$.
Thus, when needed, we are free to replace $Z$ with a surjective generically finite cover and all the relevant varieties and morphisms by base change (recall that the base change of a locally stable family is again a locally stable family).
In particular, we may assume that $Z$ is smooth.

Let $c \colon X \rar Y$ denote the relative Iitaka fibration of $(X,\Delta)$, and let $(Y,B_Y,\bM.)$ be the generalised pair induced on $Y$.
By Noetherianity and Proposition \ref{semistable_iitaka}, 
we may find finitely many effective divisors 
$D_1,\ldots,D_k \in |\bM Y.|_\QQ$ and open sets $U_i \subset Z$ such that 
$Z = \cup_{i= 1}^k U_i$ and $g \colon (Y,B_Y+D_i) \rar Z$ is a stable family of pairs over $U_i$.

Now, fix $i \in \{1, \dots, k\}$.
We will show that for integers $m_i \gg 0$ and $\ell>0$ and $\epsilon \in \QP$, 
the base locus of $\lambda_{\ell m_i}$ 
(resp. $\bM Y. +\epsilon f^*c_1(\lambda_{\ell m_i})$) is contained in $Z \setminus U_i$ 
(resp. $f^{-1}(Z \setminus U_i)$).
Supposing this, we see that by taking $m = m_1m_2\dots m_k$ the claims of the theorem hold.

For ease of notation set $D\coloneqq D_i$ and $U \coloneqq U_i$.

By \cite[Corollary 6.18]{KP17}, we may find a generically finite cover $\sigma\colon \overline{U} \rightarrow U$ such that $g \colon (Y,B_Y+D)\times_U\overline{U} \rar \overline{U}$ can 
be recompactified to a family of stable pairs over a normal projective
variety $\overline{Z} \supset \overline{U}$.
Call this compactified family $(\overline{Y}, \overline{B}_Y+\overline{D}) \rightarrow \overline{Z}$.
Perhaps replacing $\overline{Z}$ by a higher model, we may assume that $\sigma$ extends to 
a morphism $\sigma\colon \overline{Z} \rightarrow Z$.

Let $\tau\colon \overline{Z} \rightarrow \tilde{Z}$ be the Stein factorisation of $\sigma$, and set $\tilde{Y} \coloneqq Y\times_Z\tilde{Z}$.
By part (4) of Theorem~\ref{thm_iitaka}, $\tilde{Y}$ is normal.
Since $\psi \colon \tilde{Y} \to Y$ is finite, pullback is well defined also for divisors that are not necessarily $\mathbb Q$-Cartier.
We define $\tilde{B}$ via $\psi^*(K_Y+B_Y) = K_{\tilde{Y}}+\tilde{B}$ and we set $\tilde{D}=\psi^*D$.

Let $m_0$ be the integer guaranteed to exist by  Theorem~\ref{thm_iitaka}.
Up to replacing $m_0$
with a multiple, we may further assume that $f_* \scr O_X.(\ell m_0(\K X/Z.+\Delta))$ 
is a vector bundle that commutes with base change.
Note that $\det h \subs *. \scr O_{\overline{Y}}.(km_0(\K \overline Y /\overline Z. + \overline B  + \overline D ))\vert_{\overline{U}}$ 
coincides with $\sigma^*\det f_* \scr O_X.(km_0(\K X/Z.+\Delta))\vert_{U}$.

Now, we claim that 
we have an inclusion of coherent sheaves
\begin{equation} \label{inj_vb}
0 \rar h \subs *. \scr O_{\overline {Y}}.(km_0(\K \overline Y /\overline Z. + \overline B  + \overline D )) 
\rar \sigma^*f_* \scr O_X.(km_0(\K X/Z.+\Delta))
\end{equation}
that is an isomorphism over $\overline{U}$.
The latter claim is immediate by the previous observations.
Furthermore, as the two sheaves are torsion free, to conclude it suffices to show there is a non-zero morphism between the two.

Now, let $\hat{Y}$ be the normalisation of the closure of the graph of $\tilde{Y} \drar \overline{Y}$ in $\tilde{Y} \times_{\tilde Z} \overline{Y}$, and let $p \colon \hat{Y} \rar \tilde Y$ and $q \colon \hat{Y} \rar \overline{Y}$ be the corresponding morphisms.
As $\tilde Y$ and $\overline Y$ are normal and isomorphic over the generic point of $\tilde Z$, $p$ and $q$ are isomorphisms over the generic point of $\tilde Z$.
To conclude, it is then sufficient to show that
\begin{equation} \label{ineq_divisors}
q^*(\K \overline{Y}. + \overline{B} + \overline{D}) +\Theta= p^*(\K \tilde Y. + \tilde B +\tilde D),
\end{equation}
where $\Theta \geq 0$.
Note that, by construction, $p_*\Theta$ is contained in the preimage of $Z \setminus U$ in $\tilde Y$.

To this end, we follow the proof of \cite[Theorem 11.40]{Kollar21}.
Then, we may write
\[
p^*(\K \tilde Y. + \tilde B + \tilde D)= \K \hat{Y}. + \hat{B} + \hat{D} + A -B,
\]
where $\hat{B}$ and $\hat{D}$ denote the strict transforms of $\tilde B$ and $\tilde D$, respectively, 
and $A$ and $B$ are effective divisors with no common components.
Since $h \colon (\overline{Y}, \overline{B} + \overline{D}) \rar \overline{Z}$ is a stable family, we may write
\[
q^*(\K \overline{Y}. + \overline{B} + \overline{D})= \K \hat{Y}. + \hat{B} + \hat{D} -E,
\]
where $E$ is an effective divisor, cf. \cite[Proposition 2.15]{Kollar21}.
Then, it follows that $-(E+A-B)$ is $p$-nef and $p_{*}(E+A-B)=p_{*}E \geq 0$.
Thus, by the negativity lemma \cite[Lemma 3.39]{KM98}, we have $E+A-B \geq 0$.
Thus, in particular, $A-B \geq -E$, and \eqref{ineq_divisors} follows.

By taking determinants of \eqref{inj_vb}, we get a morphism
\begin{equation} \label{inj_lb}
\bar{\lambda}_{\ell m_0}\coloneqq \det h \subs *. \scr O_{\overline{Y}}.(\ell m_0(\K \overline Y /\overline Z. + \overline B  + \overline D )) 
\rar \sigma^*\det f_* \scr O_X.(\ell m_0(\K X/Z.+\Delta))
\end{equation}
that is an isomorphism over $\overline{U}$, and so
$\sigma^*c_1(\lambda_{\ell m_0}) = c_1(\bar{\lambda}_{\ell m_0})+G$
where $G \geq 0$ is supported on $\overline{Z} \setminus \overline{U}$.
Likewise, we see that 
$q^*(\overline{M}+\epsilon \bar{f}^*c_1(\bar{\lambda}_{\ell m_0}))+\Theta+\epsilon q^*\bar{f}^*G = p^*(\tilde M+\epsilon \tilde f^*c_1(\tilde{\lambda}_{\ell m_0}))$,
where $\tilde M$ and $\tilde \lambda _{\ell m_0}$ are the pull-backs of $M$ and $\lambda_{\ell m_0}$ to $\tilde X = X \times_{\tilde Z} Z$ and $\tilde Z$, respectively.

By Lemma \ref{lem_stable_max}, we know that $\bar{\lambda}_{\ell m_0}$ 
(resp. $(\overline{M}+\epsilon \bar{f}^*c_1(\bar{\lambda}_{\ell m_0}))$) is semi-ample, 
which from our above equalities implies that the base locus of $\tilde{\lambda}_{\ell m_0}$ 
(resp. $\tilde M+\epsilon \tilde f^* c_1(\tilde \lambda_{\ell m_0}))$) is contained in
$G$ (resp. $p_*(\Theta+\epsilon q^*\bar{f}^*G)$) as required.
As explained at the beginning of the proof, these statements are well behaved under base change, and thus we can descend the corresponding statements regarding
$\lambda _{\ell m_0}$ and $M + \epsilon f^* c_1(\lambda_{\ell m_0})$ to $Z$ and $X$, respectively.
This allows us to conclude.
\end{proof}

\begin{theorem} \label{main_thm_techincal}
\label{thm_*}
Assume that the 
b-semi-ampleness conjecture holds for LC trivial fibrations of relative dimension at most $n-1$.

Let $f\colon (X, B) \rightarrow Z$ be a GLC fibration with $B \geq 0$ and $\dim X= n$.
Assume $\K X. + B$ is $f$-nef and 
that a general fibre of $f$ is a good minimal model.

Then we may find a birationally equivalent GLC fibration 
$f'\colon (X', B') \rightarrow Z'$
fitting into the following commutative diagram
\[
\begin{tikzcd}
  (X', B') \arrow[r, "\alpha", dashrightarrow] \arrow[d, "f'"] & (X, B) \arrow[d, "f"] \\
  Z' \arrow[r, "\beta"] &  Z,
 \end{tikzcd}
\]
where $\alpha$ and $\beta$ are birational,
such that the following holds
\begin{enumerate}
\item $(X, B)$ and $(X', B')$ are crepant birational over the generic point of $Z$;

\item $(X'/Z', B')$ satisfies Property $(*)$, is BP stable and has maximal moduli;

\item $f\colon X' \rightarrow Z'$ is equidimensional;

\item the moduli part, $M'$, of $(X'/Z', B')$ is nef;

\item if, in addition, $(X, B^h) \to Z$ is a locally stable family, then $f'_*\mathcal O_{X'}(mM')$ is locally free and 
$\lambda_m'\coloneqq \det f'_*\mathcal O_{X'}(mM')$ is nef and semi-ample for $m \gg 0$ sufficiently divisible; and

\item if, in addition, $(X, B^h) \to Z$ is a locally stable family, then $M'+\epsilon f'^*c_1(\lambda_m')$ is semi-ample for any $\epsilon \in \QP$.
\end{enumerate}

Moreover, $\lambda_m'$ is compatible with base change in the following sense.
Consider another birationally equivalent GLC fibration $f''\colon (X'', B'') \rightarrow Z''$, 
where $\gamma\colon Z'' \rightarrow Z'$ is a birational morphism, such that $(X''/Z'', B'')$ satisfies Property $(*)$ and is BP stable,
$f''$ is equidimensional and such that the moduli part of $(X''/Z'', B'')$ is nef.
Then, we have $\lambda_m'' \cong \gamma^*\lambda_m'$.  
\end{theorem}

\begin{proof}
By \cite[Theorem 1.1]{ACSS} we may find a model $(X', B') \rightarrow Z'$ which satisfies (1)-(4).
By Lemma \ref{lem_gen_fibre_min_model} we see that that the assumption on termination of flips in 
\cite[Theorem 1.1]{ACSS} is unneeded.
Theorem \ref{thm_locally_stable_semi-ample} implies that (5) and (6) hold.

To prove our final claim, observe that $(X'', B'') \rightarrow Z''$ also satisfies
(1)-(6) and so our final claim is a consequence of Lemma \ref{lem_base_change_pushforward}.
\end{proof}

\begin{remark}
In the case where the moduli part is $f$-big, it is not necessary to assume the b-semi-ampleness conjecture.
\end{remark}

\begin{corollary}
Assume that the 
b-semi-ampleness conjecture holds for LC trivial fibrations of relative dimension at most $n-1$.

Let $f\colon (X, B) \rightarrow Z$ be an equidimensional 
GLC fibration with $B \geq 0$ and $\dim X= n$ which satisfies Property $(*)$.
Suppose that $K_X+B$ is $f$-nef and a general fibre of $f$ admits a good minimal model.

For $m \gg 0$ sufficiently divisible, set $\lambda_m \coloneqq \det f_*\mathcal O_X(mM)$.  
Then, there exist birational models $b\colon \tilde{Z} \rightarrow Z$ 
and $a\colon \tilde{X} \rightarrow X$ and a semi-ample line bundle $L$ on $\tilde{Z}$ 
such that $b_*c_1(L) = c_1(\lambda_m)$,
and, for any $\epsilon \in \QP$, there is a semi-ample $\mathbb{Q}$-Cartier divisor $N$ on $\tilde{X}$
such that $a_*N = M+\epsilon c_1(f^*\lambda_m)$.

\end{corollary}

\begin{proof}
Consider the model $(X'/Z', B')$ with moduli part $M'$ guaranteed to exist by Theorem \ref{thm_*}, 
and let $p\colon \tilde{X} \rightarrow X$ and $q\colon \tilde{X} \rightarrow X'$ resolve the 
rational map $\alpha\colon X' \dashrightarrow X$.
Similarly, let $\beta \colon Z' \rightarrow Z$ be the corresponding morphism as in Theorem \ref{thm_*}.

By \cite[Proposition 2.21]{ACSS}, we see that $p^*M \sim q^*M' + D$ for some effective divisor $D$.
By Theorem \ref{thm_*}, $(X'/Z',B')$ has maximal moduli (see \cite[Definition 2.20]{ACSS});
so $D=0$ holds.
It follows that we have $\beta_*f'_*\mathcal O_{X'}(mM')  = f_*\mathcal O_X(mM)$,
from which we may conclude that $\beta^*c_1(\lambda_m) = c_1(\lambda')+F$ where 
$\lambda' = \det f'_*\mathcal O_{X'}(mM')$ and $F$ is $\beta$-exceptional.  Since $\lambda'$ is nef
\cite[Lemma 3.39]{KM98}, we have $F \geq 0$.  Take $\tilde{Z}=Z'$ and $L = \lambda'$ to conclude.  
The claim for $M+\epsilon c_1( f^*\lambda_m)$ follows similarly.
\end{proof}


\section{$1$-dimensional bases}

\begin{proposition}
\label{prop_lambda_plus}
Let $f \colon (X,B) \rar Z$ be a locally stable family of good minimal models over a smooth curve $Z$
and let $M$ be the moduli part associated to $(X/Z, B)$.
Further assume  
that $\lambda_+>0$ and that a general fibre of $f$ is a klt pair.

Then
\begin{enumerate}
\item $M+\epsilon f^*c_1(\lambda_m)$ is semi-ample for $m\gg 0$ any $\epsilon > 0$
where $\lambda_m = \det f_*\mathcal O_X(mM)$; and

\item $\kappa(M) \geq \kappa(X/Z, B)+{\rm var}(X/Z, B)$.
\end{enumerate}
\end{proposition}

\begin{proof}
By Lemma~\ref{lem_mod_part_stable}, we have $M=K_{X/Z}+B$, which is $f$-nef by assumption.
By Proposition~\ref{prop_property*_open}, $(X/Z,B)$ satisfies Property $(*)$.
Furthermore, by \cite[Theorem 2.11]{ACSS}, it is also BP stable.
In turn, by \cite[Theorem 4.4]{ACSS}, $M$ is globally nef;
we observe that, by Lemma \ref{lem_gen_fibre_min_model}, the assumption on termination of flips in 
\cite[Theorem 4.4]{ACSS} is unneeded.

Let $(Y,B_Y, \mathbf{M})$ denote the relatively ample model of $(X, B)$ over $Z$, 
with morphisms $c \colon X \rar Y$ and $g \colon Y \rar Z$.
Then, we have that $\kappa(M_X)=\kappa(\K Y/Z. + B_Y + \mathbf{M}_Y)$, as we have $M_X=\K X/Z. + B$.

Since $(X,B) \to Z$ is a locally stable family and a general fibre is a klt pair,
then $(X,B)$ is klt.
In particular, by \cite[Theorem 0.2]{Ambro05}, we may choose $0 \leq D \sim_{\mathbb Q} \mathbf M _Y$
such that $(Y,B_Y+D)$ is a klt pair.
Then, for $m>0$ sufficiently divisible, by \cite[Proposition 6.4]{CoPa21} and the projection formula,
$E_m = f_*\mathcal{O}_X(mM)$ is a nef vector bundle.

Now, we claim that $\lambda_m = \det E_m$ is ample.
Since $E_m$ is nef, we have $c_1(\lambda_m) \geq 0$.
If $c_1(\lambda_m)=0$, then $E_m$ would be numerically flat.
On the other hand, $\lambda_+ > 0$ implies that $M-\delta f^*A$ is $\mathbb Q$-effective,
where $A$ is an ample divisor on $Z$ and $0 < \delta \ll 1$.
It follows that for $m>0$ sufficiently divisible $H^0(Z, E_m \otimes \mathcal O(-A)) \neq 0$ and so we have a non-zero morphism $\mathcal O(A) \to E_m$.  This implies that $0<\mu(\mathcal O(A)) \leq \mu_{{\rm max}}(E_m)$, in particular $E_m$ is not numerically flat, and so 
$\lambda_m$ is ample.

Item (1) is a direct consequence of the fact that $\lambda_m$ is ample
and $M=K_{X/Z}+B$ is nef and $f$-semi-ample.
Indeed, $M$ is the pull-back of a nef and $g$-ample $\mathbb{Q}$-divisor $L$.
Then, by \cite[Proposition 1.7.10]{Lazarsfeld04a}, $\eta L + \epsilon g^*c_1(\lambda_m)$ is ample for $0 < \eta \ll \epsilon$;
then, since $L$ is globally nef, we can add any positive multiple of $L$ to $\eta L + \epsilon g^*c_1(\lambda_m)$ and preserve ampleness.
Then, by pull-back to $X$, the claim follows.

By dimension reasons, item (2) is implied by the bigness of  $\K Y/Z. + B_Y + \mathbf{M}_Y$.
By construction, $\K Y/Z. + B_Y + \mathbf{M}_Y$ is ample over $Z$.
Furthermore, since $\lambda_+>0$, we may find an ample divisor $H$ on $Z$ such that 
$\K Y/Z. + B_Y + \mathbf{M}_Y - g^* H$ is pseudo-effective.
Now, since $H$ is ample and $\K Y/Z. + B_Y + \mathbf{M}_Y$ is $g$-ample, for $0 < \epsilon \ll 1$, 
we have that $\epsilon(\K Y/Z. + B_Y+\mathbf{M}_Y) + g^*H$ is ample.
And so we have we have that
\[
\epsilon(\K Y/Z. + B_Y + \mathbf{M}_Y)+g^*H + (\K Y/Z. + B_Y + \mathbf{M}_Y - g^* H) = 
(1+\epsilon)(\K Y/Z. + B_Y +\mathbf{M}_Y)
\]
is big as required.
\end{proof}

\begin{lemma}
\label{lemma_product}
Let $X$ be a normal projective variety,  let $Z$ be normal quasi-projective variety 
and let $Y$ be a normal variety which is projective 
over $Z$.
\begin{enumerate}
\item Suppose that we have a surjective contraction $\phi\colon X\times Z \rightarrow Y$ over $Z$.
Then there is a morphism $\varphi\colon X \rightarrow X'$ and an isomorphism $Y \cong X'\times Z$.

\item Let $D$ be a divisor on $Y=X\times Z$ and suppose that the $D$-flip (resp. $D$-flop) over $Z$, $Y^+$, exists.  
Then there exists a divisor $D_0$ on $X$ such that the $D_0$-flip (resp. $D_0$-flop), $X^+$, exists and
$Y^+ \cong X^+\times Z$.
\end{enumerate}

\end{lemma}
\begin{proof}
To see (1) take $X' = \phi(X \times \{z\})$ for any $z \in Z$ and set $\varphi = \phi\vert_{X\times \{z\}}$. 
Let $A'$ be an ample divisor on $X'$ and
let $A_Y$ be an ample divisor on $Y$.
We observe that $(\phi^*A_Y)^\perp \cap \overline{NE}(X\times Z) 
= (\phi^*A_Y)^\perp \cap \overline{NE}(X\times Z/Z) \cong (\varphi^*A')^\perp \cap \overline{NE}(X)$, 
where the first equality
holds because $\phi$ is a morphism over $Z$.
It follows that the morphisms over $Z$ defined by $\phi^*A_Y$ and $\pi^* \circ \varphi^*A'$,
where $\pi$ denotes the projection $X \times Z \to Z$,
contract the same curves, and so by the rigidity lemma, 
\cite[Lemma 1.15]{MR1841091},
they are the same morphism.

To prove (2), let $W$ be the base of the $D$-flipping (resp. $D$-flopping) contraction.  
Arguing as in the previous case
if $D_0 = D \vert_{X\times \{z\}}$ for a general point $z \in Z$ then 
$D_0 \times Z \equiv D/W$.  It follows that the $D$-flip (resp. $D$-flop)
is isomorphic to the $D_0\times Z$-flip (resp. $D_0\times Z$-flop), which is easily seen to be $X^+\times Z$, as required. 
\end{proof}

\begin{lemma} \label{lemma_crepant_product}
Let $C$ be a smooth projective variety, and let $(X_0,\Delta_0)$ be a projective klt pair with $\Delta_0 \geq 0$ and such that 
$\K X_0. +\Delta_0$ is nef.
Let $(X,\Delta)$ denote the product $(X_0,\Delta_0) \times C$, and let $f \colon (X,\Delta) \rar C$ 
denote the induced morphism.
Let $f' \colon (X',\Delta') \rar C$ be a birational model of $(X,\Delta)$, where $\Delta' \geq 0$, $(X', \Delta')$ is klt and
$K_{X'}+\Delta'$ is nef.
Then, there exists a crepant model $(X'_0,\Delta'_0)$ of $(X_0,\Delta_0)$ such that $(X',\Delta') = (X'_0,\Delta'_0)\times C$.
\end{lemma}

\begin{proof}
By taking a $\mathbb Q$-factorial 
terminalisation of $(X_0,\Delta_0)$,
which exists by \cite[Corollary 1.4.3]{BCHM06}, we may further assume 
that $(X,\Delta) \drar (X',\Delta')$ is a rational contraction and 
$X$ and $X_0$ are $\mathbb Q$-factorial.
Taking a $\mathbb Q$-factorial terminalisation $(X'',\Delta'') \rar (X',\Delta')$, 
which exists by \cite[Corollary 1.4.3]{BCHM06}, we may assume
$X \drar X''$ is an isomorphism in codimension 1.

By \cite{Kawamata07} $X \drar X''$ is a sequence of $(K_X+\Delta)$-flops.
By Lemma \ref{lemma_product} we see that if $X_0 \dashrightarrow X''_0$ is 
the induced birational contraction
on a general fibre, then in fact $X'' \cong X''_0\times C$.  
We conclude by another application of Lemma \ref{lemma_product} to $(X'', \Delta'') \rar (X', \Delta')$.
\end{proof}

\begin{definition}
\label{var_circ}
Given a GLC fibration $(X, B) \rightarrow Z$ with $B \geq 0$ and $\dim Z = 1$ 
we define ${\rm var}(X/Z, B)$ to be $0$ if there exists a pair $(X_0, B_0)$
and a generically finite morphism $Z' \rightarrow Z$ such that
$(X_0, B_0)\times Z'$ is birational to $(X, B)\times_ZZ'$.  Otherwise
we define ${\rm var}(X/Z, B)$ to be 1.
\end{definition}

\begin{remark}
\label{rem_variation}
\begin{enumerate}
\item Note that ${\rm var}(X/Z, B)$ is invariant by 
crepant birational transformations of $(X, B)$.

\item If $X$ has canonical singularities and $B = 0$ this agrees with the definition
of variation found in \cite{Kawamata85a}. 

\item A family of terminal minimal models with $B = 0$ and 
${\rm var} = 0$ is always (up to base change)
crepant birational to a product, and so by Lemma \ref{lemma_crepant_product} 
is in fact a product.

\end{enumerate}

\end{remark}

\begin{theorem}
\label{thm_curve_base}
Let $f\colon (X, B) \rightarrow Z$ be a GLC fibration between projective varieties 
with $B \geq 0$, $\dim X= n$ and $\dim Z = 1$.
Suppose that 
$(X, B)$ is klt over the generic point of $Z$ and the generic fibre of $f$ admits a good minimal model.

Then we may find a birationally equivalent GLC fibration 
$f'\colon (X', B') \rightarrow Z$
fitting into the following commutative diagram
\[
\begin{tikzcd}
  (X', B') \arrow[r, "\alpha", dashrightarrow] \arrow[d, "f'"] & (X, B) \arrow[d, "f"] \\
  Z \arrow[r, equal] &  Z.
 \end{tikzcd}
\]
such that $K_{X'}+B'$ is nef over $Z$ and 

\begin{enumerate}
\item $(X'/Z', B')$ is BP stable and has maximal moduli;
\item $\kappa(M') \geq \kappa(X'/Z, B')+{\rm var}(X'/Z, B')$; and 

\item if, in addition, $(X, B^h) \to Z$ is a locally stable family, then  $M'+\epsilon f'^*c_1(\lambda_m')$ is semi-ample for $m\gg 0$ and any $\epsilon \in \QP$
where $\lambda_m' = \det f'_*\mathcal O_{X'}(mM')$.

\end{enumerate}
\end{theorem}

\begin{remark}
We remark that $M'$ and $\lambda'_m$ are compatible with any further base change, as in 
the statement of Theorem \ref{thm_*}.
\end{remark}

\begin{proof}
Take $(X', B')$ to be the model of $(X, B)$ guaranteed by \cite[Theorem 1.1]{ACSS}.
By Lemma \ref{lem_gen_fibre_min_model} we see that that the assumption on termination of flips in 
\cite[Theorem 1.1]{ACSS} is unneeded.

Taking a base change along the finite morphism guaranteed by Proposition \ref{prop_equi*_cover}
does not alter $\kappa(M'), \kappa(X/Z, B), {\rm var}(X/Z, B)$ or (in the case $(X, B^h) \to Z$ is locally stable) the semi-ampleness of 
$M'+\epsilon f'^*c_1(\lambda_m')$, so we may freely replace
$(X', B')$ by this cover, and therefore may assume that $(X', B')$ is a locally stable family.
For ease of notation we replace $(X, B)$ by $(X', B')$.
By \cite[Theorem 1.1]{HX11}, 
we see that $K_{X}+B$ is semi-ample over $Z$.
Next, observe that since $(X, B)$ is locally stable and klt over the generic point of $Z$ that
$(X, B)$ is in fact klt, see Lemma \ref{lem_img_lc_centre_is_lc}.

Let $(\tilde{X}, \tilde{B}) \rightarrow (X, B)$ be a $\mathbb Q$-factorial terminalisation of $(X, B)$, 
which exists by \cite[Corollary 1.4.3]{BCHM06}.
By Proposition \ref{prop_kollar_criterion}, $(\tilde{X}, \tilde{B})$ is locally stable. 
So we may freely replace $(X, B)$ by 
$~(\tilde{X}, \tilde{B})$, and so we may assume 
that $X$ is $\QQ$-factorial with terminal singularities.

By Proposition \ref{prop_lambda_plus}, we may assume that $\lambda_+ = 0$.  By \cite[Proposition 4.1]{ACSS} 
we see that $K_{X/Z}+B$ is nef and so $\lambda_- \geq 0$.
For $m \gg 0$ Theorem \ref{thm_iitaka} implies that $E_m \coloneqq f_*\mathcal O_X(m(K_{X/Z}+B)$ is a vector bundle and by 
Proposition \ref{prop_threshold_to_slope} $\mu_{\min}(E) = \mu_{\max}(E) = 0$ and 
so $E_m$ is numerically flat.

Semi-ampleness of $M$ is a direct consequence of Lemma \ref{lem_flat_to_isotrivial}, 
and so item (3) is proven.

To prove item (2), let $Y$ be the relatively ample model of $(X, B)$ over $Z$ and let $(Y, B_Y, \mathbf{M})$
be the generalised pair given by the canonical bundle formula. 
By Lemma \ref{lem_flat_to_isotrivial} 
after replacing 
$Z$ by an \'etale cover we may assume $(Y, B_Y+\mathbf{M}_Y) = Z \times (Y_0, B_{Y_0}+\mathbf{M}_{Y_0})$.
As we are assuming that $\lambda_+ = 0$ it suffices to show that ${\rm var}(X/Z, B) = 0$ 
to conclude.
So we will show that after a generically finite base change, 
$(X, B) \rightarrow Z$
becomes isomorphic to a product $(X_0, B_0)\times Z$.  
We will argue by induction on $\dim (X/Z)$.
The case where $\dim (X/Z) = 0$ is obvious.

We proceed with a case by case analysis.

{\bf Case (1):} In this case we assume that ${\rm Supp}(B)$ is vertical over $Y$. 

First, we run a $K_X$-MMP with scaling relatively over $Y$.
By \cite[Theorem 1.1]{HX11} and the assumption that $B$ is vertical over $Y$, the 
MMP terminates with a good minimal model.
We denote by $X'$ the final model of this MMP and $Y'$ the relatively ample model, which is a birational model of $Y$.
Since $K_X+B\sim_{\QQ}0/Y$, $(X',B') \rar Z$ is still a family of good minimal models, 
where $B'$ denotes the push-forward of $B$ along the MMP.  Note that with respect to $(X', B') \rightarrow Z$ 
we still have $\lambda_+ = 0$, and so we may freely replace $(X, B)$ by $(X', B')$ and therefore may assume
in addition that $K_X \sim_{\mathbb Q} 0/Y$.

Now, if $B = 0$, we may conclude by \cite[Theorem 1.1]{Kawamata85a}, see also Remark \ref{rem_variation}.
Thus, we may assume that $B \neq 0$, and we wish to reduce to \cite{Kawamata85a}.
Now, consider a general ample divisor $H_0 \subset Y_0$, and let $H$ denote 
the divisor $H_0 \times Z \subset Y$.
Furthermore, we let $D$ denote the pull-back of $H$ to $X$.
For some $m \gg 0$ sufficiently large, and for general enough $H$,
we may find a cyclic Galois cover 
$\tilde{Y} \rightarrow Y$ and $\tilde{X} \rightarrow W$
of degree $=m$, which are ramified along $H$ and $D$, respectively.

Since $K_X \sim_{\QQ} 0 /Y$ and $H_0$ is sufficiently ample, 
$\tilde{X} \rar \tilde{Y}$ is the relative Iitaka fibration of $\tilde{X}$ over $Z$.
By construction, the generalised pair induced by $\tilde{X}$ on $\tilde{Y}$ is
$(\tilde Y,B_{\tilde{Y}},\tilde{\mathbf{M}})$, where $B_{\tilde{Y}}$ and $\tilde{\mathbf{M}}$ are the pull-backs
(the latter as b-divisor) of $B_Y$ and $\mathbf{M}$, respectively.
Then, for every closed point $z \in Z$, $(\tilde{Y}, B_{\tilde{Y}}+\tilde{Y}_z,\tilde{\mathbf{M}})$
is generalised log canonical, because so is $(\tilde Y, B_{ Y}+ Y _z,{\mathbf{M}})$.
Then, by inversion of adjunction for LC-trivial fibrations \cite[Theorem~3.1]{Ambro04},
we deduce that $\tilde{X} \rar Z$ is a locally stable family.
By construction, the variation of the relatively ample model of $\tilde{X}$ is 0, 
since both $(Y,B_Y+M_Y)$ and $H$ have no variation.
Then, by \cite[Theorem 1.1]{Kawamata85a}, it follows that $\tilde{X} \rar Z$ is isotrivial, and hence (up to a base change)
a product.
We observe that the requirements on the class of singularities to apply \cite[Theorem 1.1]{Kawamata85a} is actually satisfied.
Indeed, by our initial reductions, $X$ is terminal, while $H$ is a general member of a free linear system.
Thus, $H$ is itself canonical, and so is the pair $(X,H)$.
In turn, by \cite[Proposition~5.20]{KM98} and the Riemann--Hurwitz formula, the pair $(\tilde X, \tilde H)$ is canonical.
In particular, $\tilde X$ has canonical singularities.

Write $\tilde{X} = \tilde{X}_0 \times Z$, and $\tilde{Y} = \tilde{Y}_0 \times Z$.  
Let $g$ be a generator of the Galois action.
By construction, $g$ acts on $\tilde{Y}$ as $g' \times 1_Z$ for some $g' \in {\rm Aut}(\tilde{Y}_0)$.
Then, the action of $g$ on $\tilde{X}$ is a lift of this action;
in particular, by this compatibility, $g$ has to act on $\tilde{X}$ as $g'' \times 1_Z$ for some $g'' \in {\rm Aut}(\tilde{X}_0)$.
 It follows that we have an isomorpism 
 $\tilde{X}/\langle g \rangle \cong \tilde{X}_0/\langle g'' \rangle \times Z$, and so $X$ splits as a product.
 We then have $(X,B) \rar Z$ is a product, 
since we reduced to the case when $B$ is the pull-back of a divisor supported on $B_Y$.

{\bf Case (2):} In this case we assume that $B$ is big over $Y$.

Suppose for sake of contradiction
that ${\rm var}(X/Z, B) = 1$.

Let $(X', (1+\epsilon)B')$ 
denote the log canonical model of $(X, (1+\epsilon)B)$ over $Y$ for $0<\epsilon \ll 1$.  
Note that $(X', (1+\epsilon)B')$ 
is a stable family over some open subset $U \subset Z$.
Let $(X'', (1+\epsilon)B'')$ be a compactification of this stable family over $U$
to a stable family over $Z$ (perhaps after replacing $Z$ by a finite cover).
By \cite[Corollary 8.3]{KP17} and our assumption that ${\rm var}(X/Z, B) = 1$ 
we see that $K_{X''/Z}+(1+\epsilon)B''$ is big.

For general $y \in Y_0$, let $Z_y$ denote the fibre of $Y_0 \times Z \rightarrow Y_0$,
and let $(X_y'', (1+\epsilon)B''_y)$ denote the restricted family over $Z_y$.
We note that by \cite[Proposition 2.13]{Kollar21} $(X''_y, (1+\epsilon)B''_y)$ 
is locally stable over $Z_y$.  
By
Proposition \ref{prop_lift} we have that $(X''_y, (1+\epsilon)B''_y) 
\rightarrow Z_y$ is isotrivial
over $U$, hence isotrivial.
This implies that 
$K_{X''_y}+(1+\epsilon)B''_y$ is not big.
Finally, observe that $(K_{X''/Z}+(1+\epsilon)B'')\vert_{X''_y} \sim K_{X''_y/Z_y}+(1+\epsilon)B''_y$,
which contradicts the bigness of $K_{X''/Z}+(1+\epsilon)B''$, 
giving our required contradiction.

The map from $Z$ to the moduli space parametrising the fibres
of $(X'', (1+\epsilon)B'') \rightarrow Z$ is therefore constant, and so
$(X'', (1+\epsilon)B'')$ splits as a product after a generically finite base change, 
hence by Lemma \ref{lemma_crepant_product} $(X, B)$
itself splits as a product.

{\bf Case (3):} In this case we assume that $K_X$ is not pseudo-effective over $Y$ 
and we reduce to Case (2).

We run a $K_X$-MMP with scaling of any ample divisor over $Y$.
Since $K_X$ is not pseudo-effective, this MMP terminates with a Mori fibre space 
$X' \rar W$ over $Y$.
In particular, we have $\dim W < \dim X$, and hence $X' \rar W$ is not birational.
Since $K_X+B \sim_{\QQ}0/Y$, the final outcome $(X',B') \rar Z$ is still a locally stable 
family of good minimal models, and it is crepant birational to $(X,B)$.
Thus, up to replacing $(X,B)$ with $(X',B')$, we may assume that $X \rar W$ is a morphism.
If $\dim W = \dim Y$, then $W \rar Y$ is birational and $B$ is big over $Y$ 
and so we may reduce to Case (2).
Therefore, we may assume that $\dim Y < \dim W < \dim X$.

Let $(W,B_W,\mathbf{M}_W)$ denote the generalised pair induced by the canonical bundle formula 
for $(X,B) \rar W$.
By \cite{Ambro05}, we may choose $0 \leq \Gamma \sim_\QQ \mathbf{M}_W$ such that $(W,B_W+\Gamma)$ is klt.
It follows that $(W,B_W+\Gamma) \rar Z$ is generically a locally stable family of good minimal models, and so, 
up to base change of $Z$, we may assume that this family admits a compactification 
$(W',B_{W'}+\Gamma') \rar Z$ that is everywhere locally stable and that agrees with 
$(W,B_W+\Gamma)\rar Z$ over a non-empty open subset of $Z$.
Let $\tilde W$ denote the normalisation of the graph of $W \rar W'$, and let 
$p \colon \tilde W \rar W$ and $q \colon \tilde W \rar W'$ denote the corresponding morphisms.
Following the proof of \cite[Theorem 11.40]{Kollar21}, it follows that
\begin{equation}\label{ineq_discrep}
p^*(K_W+B_W+\Gamma) \geq q^*(K_{W'}+B_{W'}+\Gamma').
\end{equation}
Since $(W,B_W+\Gamma)$ is crepant to $(X,B)$, it follows that $\lambda_-=\lambda_+=0$ for 
$(W,B_W+\Gamma) \rar Z$.
Then, by \eqref{ineq_discrep} and the fact that $(W',B_{W'}+\Gamma')\rar Z$ is a locally stable 
family of good minimal models, it follows that $\lambda_-=\lambda_+=0$ for $(W',B_{W'}+\Gamma')\rar Z$.
By our induction hypothesis, up to a base change, $(W',B_{W'}+\Gamma')\rar Z$ splits as a product.

Since $(Y,B_Y+\mathbf{M}_Y)$ and $(W',B_{W'}+\Gamma')$ are products, and 
since for a general $z$ we have a morphism
$g_z\colon W'_z \rightarrow Y_z$ with 
$g_z^*(K_{Y, z}+B_{Y, z}+\mathbf{M}_{Y, z}) \sim_{\mathbb Q} K_{W', z}+B_{W', z}+\Gamma'_z$
it follows that there exists a morphism $g\colon W' \rightarrow Y$ such that
$g^*(K_Y+B_Y+\mathbf{M}_Y) \sim_{\mathbb Q} K_{W'}+B_{W'}+\Gamma'$.
By construction we have a morphism $h\colon W \rightarrow Y$ such that 
$K_{W}+B_W+\Gamma \sim_{\mathbb Q} h^*(K_Y+B_Y+\mathbf{M}_Y)$
and so 
$(W,B_W+\Gamma)$, 
and $(W',B_{W'}+\Gamma')$ are crepant birational to each other.
By Lemma \ref{lemma_crepant_product} it follows that $(W, B_W+\Gamma)$ is a product.

We may then argue as in Case (2), applied to the morphism $(X, B) \rightarrow W$ 
to deduce that $(X, B)$ is itself a product.
\end{proof}

\bibliography{math.bib}
\bibliographystyle{alpha}

\end{document}